\documentclass[a4paper,11pt]{article}

\usepackage[utf8]{inputenc}
\usepackage[english]{babel}
\usepackage{mleftright}
\usepackage[inline,shortlabels]{enumitem}
\usepackage{graphicx}
\usepackage{subfig}
\usepackage{microtype}
\usepackage{amsmath}
\usepackage{amssymb}
\usepackage{amsthm}
\usepackage{cite}
\usepackage[a4paper,total={6in,9in}]{geometry}
\usepackage[colorlinks,allcolors=red]{hyperref}
\usepackage[capitalize,noabbrev]{cleveref}

\setlist{font=\normalfont}

\theoremstyle{definition}
\newtheorem{definition}{Definition}[section]
\newtheorem{algorithm}[definition]{Algorithm}
\newtheorem{assumption}[definition]{Assumption}
\newtheorem{example}[definition]{Example}

\theoremstyle{plain}
\newtheorem{lemma}[definition]{Lemma}
\newtheorem{theorem}[definition]{Theorem}
\newtheorem{proposition}[definition]{Proposition}
\newtheorem{corollary}[definition]{Corollary}

\newcommand{\ConX}{C}
\newcommand{\SpaceY}{Y}
\newcommand{\ConY}{K}
\newcommand{\SpaceH}{H}
\newcommand{\ConH}{\mathcal{K}}
\newcommand{\Feas}{\Phi}
\newcommand{\Subseq}{I}

\newcommand{\R}{\mathbb{R}}
\newcommand{\N}{\mathbb{N}}
\newcommand{\BddMul}{w}
\newcommand{\Lag}{\mathcal{L}}
\newcommand{\wto}{\rightharpoonup}
\newcommand{\Mto}{\overset{M}{\vphantom{\rule{0pt}{0.4ex}}\smash{\,\longrightarrow\,}}}
\newcommand{\dist}[1]{\operatorname{dist}\bigl( #1 \bigr)}

\newcommand{\RadCone}[2]{\mathcal{R}_{#1}(#2)}
\newcommand{\TngCone}[2]{\mathcal{T}_{#1}(#2)}
\newcommand{\NorCone}[2]{\mathcal{N}_{#1}(#2)}

\title{
    Quasi-Variational Inequalities in Banach Spaces: Theory and Augmented Lagrangian Methods%
    \thanks{This research was supported by the German Research Foundation (DFG) within the priority
    program ``Non-smooth and Complementarity-based Distributed Parameter Systems: Simulation and
    Hierarchical Optimization'' (SPP 1962) under grant number KA 1296/24-1.}
}

\date{December 3, 2018}

\author{
    Christian Kanzow$^{\dagger}$ \and Daniel Steck%
    \thanks{University of W\"urzburg, Institute of Mathematics, Campus Hubland Nord, Emil-Fischer-Str.\ 30, 97074 Würzburg, Germany; \{kanzow,daniel.steck\}@mathematik.uni-wuerzburg.de}
}

\begin{document}

\maketitle

{
\small\textbf{\abstractname.}
This paper deals with quasi-variational inequality problems (QVIs) in a generic Banach space setting. We provide a theoretical framework for the analysis of such problems which is based on two key properties: the pseudomonotonicity (in the sense of Brezis) of the variational operator and a Mosco-type continuity of the feasible set mapping. We show that these assumptions can be used to establish the existence of solutions and their computability via suitable approximation techniques. In addition, we provide a practical and easily verifiable sufficient condition for the Mosco-type continuity property in terms of suitable constraint qualifications.

Based on the theoretical framework, we construct an algorithm of augmented Lagrangian type which reduces the QVI to a sequence of standard variational inequalities. A full convergence analysis is provided which includes the existence of solutions of the subproblems as well as the attainment of feasibility and optimality. Applications and numerical results are included to demonstrate the practical viability of the method.
\par\addvspace{\baselineskip}
}

{
\par\addvspace{\baselineskip}
}

{
\par\addvspace{\baselineskip}
}

\section{Introduction}

Let $X$ be a real Banach space with continuous dual $X^*$, and let $F:X\to X^*$, $\Feas:X\rightrightarrows X$ be given mappings. The purpose of this paper is to analyze the \emph{quasi-variational inequality (QVI)} which consists of finding $x\in X$ such that
\begin{equation}\label{Eq:QVI}
    x\in \Feas(x), \quad \langle F(x),d\rangle\ge 0
    \quad \forall d\in \TngCone{\Feas(x)}{x},
\end{equation}
where $\TngCone{\Feas(x)}{x}$ is the (Bouligand) tangent cone to $\Feas(x)$ at $x$, see Section~\ref{Sec:Prelims}. If $\Feas$ is \emph{convex-valued}, i.e., if $\Feas(x)$ is convex for all $x$, then \eqref{Eq:QVI} can equivalently be stated as
\begin{equation}\label{Eq:QVI_Convex}
    x\in \Feas(x), \quad \langle F(x),y-x\rangle\ge 0
    \quad \forall y\in \Feas(x).
\end{equation}
The QVI was first defined in \cite{Bensoussan1984} and has since become a standard tool for the modeling of various equilibrium-type scenarios in the natural sciences. The resulting applications include game theory \cite{Harker1991}, solid and continuum mechanics \cite{Beremlijski2002,Haslinger1984,Kravchuk2007,Outrata1998}, economics \cite{Ichiishi1983,Jing1999}, probability theory \cite{Kharroubi2010}, transportation \cite{Bliemer2003,DeLuca1992,Scrimali2004}, biology \cite{Hammerstein1994}, and stationary problems
in superconductivity, thermoplasticity, or electrostatics \cite{Hintermueller2012,Hintermueller2013a,Kunze2000,Rodrigues2000,Alphonse2018}. For further information, we refer the reader to the corresponding papers, the monographs \cite{Baiocchi1984,Kravchuk2007,Mosco1976}, and the references therein.

Despite the abundant applications of QVIs, there is little in terms of a general theory or algorithmic approach for these problems, particularly in infinite dimensions. The present paper is an attempt to fill this gap. Observe that, in most applications, the feasible set mapping $\Feas$ which governs the QVI can be cast into the general framework
\begin{equation}\label{Eq:FeasibleSet}
    \Feas(x)=\{ y\in \ConX: G(x,y)\in\ConY \},
\end{equation}
where $\ConX\subseteq X$ and $\ConY\subseteq\SpaceY$ are nonempty closed convex sets, $Y$ is a real Banach space, and $G:X^2\to\SpaceY$ a given mapping. One can think of $\ConX$ as the set of ``simple'' or non-parametric constraints, whereas $G$ describes the parametric constraints which turn the problem into a proper QVI. This decomposition is of course not unique.

In this paper, we provide a thorough analysis of QVIs in general and for the specific case where the feasible set is given by \eqref{Eq:FeasibleSet}. Our analysis is based on two key properties: the pseudomonotonicity of $F$ (in the sense of Brezis) and a weak sequential continuity of $\Feas$ involving the notion of Mosco convergence (see \cref{Sec:Prelims}). These properties provide a general framework which includes many application examples, and they can be used to establish multiple desirable characteristics of QVIs such as the existence of solutions and their computability via suitable approximation techniques. In addition, we provide a systematic approach to Karush--Kuhn--Tucker(KKT)-type optimality conditions for QVIs, and give a sufficient condition for the Mosco-type continuity of $\Feas$ in terms of suitable constraint qualifications. To the best of our knowledge, this is one of the first generic and practically verifiable sufficient conditions for Mosco convergence in the QVI literature.

After the theoretical investigations, we turn our attention to an algorithmic approach based on augmented Lagrangian techniques. Recall that the augmented Lagrangian method (ALM) is one of the standard approaches for the solution of constrained optimization problems and is contained in almost any textbook on optimization \cite{Bertsekas1982,Bertsekas1995,Nocedal2006,Conn2000}. In recent years, ALMs have seen a certain resurgence \cite{Andreani2007,Birgin2014} in the form of safeguarded methods which use a slightly different update of the Lagrange multiplier estimate and turn out to have very strong global convergence properties \cite{Birgin2014}. A comparison of the classical ALM and its safeguarded analogue can be found in \cite{Kanzow2017}. Moreover, the safeguarded ALM has been extended to quasi-variational inequalities in finite dimensions \cite{Pang2005,Kanzow2016,Kanzow2017a} and to constrained optimization problems and variational inequalities in Banach spaces \cite{Kanzow2016b,Kanzow2017b,Kanzow2017c,Karl2018}.

In the present paper, we will continue these developments and present a variant of the ALM for QVIs in the general framework described above. Using the decomposition \eqref{Eq:FeasibleSet} of the feasible set, we use a combined multiplier-penalty approach to eliminate the parametric constraint $G(x,y)\in\ConY$ and therefore reduce the QVI to a sequence of standard variational inequalities (VIs) involving the set $\ConX$. The resulting algorithm generates a primal-dual sequence with certain asymptotic optimality properties.

The paper is organized as follows. In Section~\ref{Sec:Prelims}, we establish and recall some theoretical background for the analysis of QVIs, including elements of convex and functional analysis. Section~\ref{Sec:KKT} contains a formal approach to the KKT conditions of the QVI, and we continue with an existence result for solutions in Section~\ref{Sec:Existence}. Starting with Section~\ref{Sec:Method}, we turn our attention to the augmented Lagrangian method. After a formal statement of the algorithm, we continue with a global convergence analysis in Section~\ref{Sec:ConvexConv} and a primal-dual convergence theory for the nonconvex case in Section~\ref{Sec:GeneralConv}. We then provide some applications of the algorithm in Section~\ref{Sec:Applic} and conclude with some final remarks in Section~\ref{Sec:Final}.

\textbf{Notation.} Throughout this paper, $X$ and $\SpaceY$ are always real Banach spaces, and their duals are denoted by $X^*$ and $\SpaceY^*$, respectively. We write $\to$, $\wto$, and $\wto^*$ for strong, weak, and weak-$^*$ convergence, respectively, and denote by $B_r^X$ the closed $r$-ball around zero in $X$ (similarly for $\SpaceY$, etc.). Moreover, partial derivatives are denoted by $D_x,D_y$, and so on.

\section{Preliminaries}\label{Sec:Prelims}

We begin with some preliminary definitions. If $S$ is a nonempty closed subset of some space $Z$, then $S^{\circ}:=\{ \psi\in Z^*: \langle \psi,s \rangle\le 0~ \forall s\in S \}$ is the \emph{polar cone} of $S$. Moreover, if $z\in S$ is a given point, we denote by
\begin{equation*}
    \TngCone{S}{z}:=\bigl\{ d\in Z: \exists\, z^k\to z,~t_k\downarrow 0 \text{ such that } z^k\in S
    \text{ and } (z^k-z)/t_k\to d \bigr\}
\end{equation*}
the \emph{tangent cone} of $S$ in $z$. If $S$ is additionally convex, we also define the \emph{radial} and \emph{normal cones}
\begin{equation*}
    \RadCone{S}{z}:=\bigl\{ \alpha (s-z): \alpha\ge 0,~s\in S \bigr\}, \quad
    \NorCone{S}{z}:=\bigl\{ \psi\in Z^*: \langle \psi,s-z \rangle\le 0~
    \forall s\in S \bigr\}.
\end{equation*}
For convex sets, it is well-known that $\TngCone{S}{z}=\overline{\RadCone{S}{z}}$ and $\NorCone{S}{z}=\TngCone{S}{z}^{\circ}=\RadCone{S}{z}^{\circ}=(S-z)^{\circ}$. Clearly, if $Z$ is a Hilbert space, we may treat $\NorCone{S}{z}$ and other polars of sets in $Z$ as subsets of $Z$ instead of $Z^*$.

Throughout this paper, we will also need various notions of continuity.

\begin{definition}\label{Dfn:Continuities}
    Let $X,Y$ be Banach spaces and $T:X\to Y$ an operator. We say that $T$ is
\begin{enumerate}[(a)]
\item \emph{bounded} if it maps bounded sets to bounded sets;
\item \emph{weakly sequentially continuous} if $x^k\wto x$ implies $T(x^k)\wto T(x)$;
\item \emph{weak-$^*$ sequentially continuous} if $Y=W^*$ for some Banach space $W$, and $x^k\wto x$ implies $T(x^k)\wto^* T(x)$ in $W^*$;
\item \emph{completely continuous} if $x^k\wto x$ implies $T(x^k)\to T(x)$.
\end{enumerate}
\end{definition}

Clearly, complete continuity implies ordinary continuity as well as weak sequential continuity. Moreover, the latter implies weak-$^*$ sequential continuity.

\subsection{Weak Topologies and Mosco convergence}\label{Sec:PrelimsWeakTop}

For the treatment of the space $X$, we will need some topological concepts from Banach space theory. The \emph{weak topology on $X$} is defined as the weakest (coarsest) topology for which all $f\in X^*$ are continuous \cite{Brezis2011}. Accordingly, we call a set weakly closed (open, compact) if it is closed (open, compact) in the weak topology. Note that the notion of convergence induced by the weak topology is precisely that of weak convergence \cite{Rudin1991}.

In many cases, e.g., when applying results from the literature which are formulated in a topological framework, it is desirable to use sequential notions of closedness, continuity, etc.\ instead of their topological counterparts. For instance, in the special case of the weak topology, it is well-known that every weakly closed set is weakly sequentially closed, but the converse is not true in general \cite{Bauschke2011}.

A possible remedy to this situation is to define a slightly different topology, which we call the \emph{weak sequential topology on $X$}. This is the topology induced by weak convergence; more precisely, a set is called open in this topology if its complement is weakly sequentially closed. It is easy to see that this induces a topology (see \cite{Birkhoff1936,Dudley1964}). Moreover, since every weakly closed set is weakly sequentially closed, the weak sequential topology is finer than the weak topology, and like the latter \cite[Prop.~3.3]{Brezis2011} it is also a Hausdorff topology. We will make use of the weak sequential topology in \cref{Sec:KKT}.

%
%

For the analysis of the QVI \eqref{Eq:QVI}, we will inevitably need certain continuity properties of the set-valued map $\Feas:X\rightrightarrows X$. Since we are dealing with a possibly infinite-dimensional space $X$, these properties should take into account the weak (sequential) topology on $X$. An important notion in this context is that of Mosco convergence.

\begin{definition}[Mosco convergence]\label{Dfn:MoscoConvergence}
    Let $S$ and $S_k$, $k\in\N$, be subsets of $X$. We say that $\{S_k\}$ \emph{Mosco converges} to $S$, and write $S_k\Mto S$, if
\begin{enumerate}[(i)]
\item for every $y\in S$, there is a sequence $y^k\in S_k$ such that $y^k\to y$, and
\item whenever $y^k\in S_k$ for all $k$ and $y$ is a weak limit point of $\{y^k\}$, then $y\in S$.
\end{enumerate}
\end{definition}

The concept of Mosco convergence plays a key role in multiple aspects of the analysis of QVIs such as existence \cite{Mosco1976,Kunze2000}, approximation \cite{Lignola1997}, or the convergence of algorithms \cite{Hintermueller2012}. It is typically used as part of a continuity property of the mapping $\Feas$.

\begin{definition}[Weak Mosco-continuity]\label{Dfn:MoscoContinuity}
    Let $\Feas:X\rightrightarrows X$ be a set-valued mapping and $x\in X$. We say that $\Feas$ is \emph{weakly Mosco-continuous in $x$} if $x^k\wto x$ implies $\Feas(x^k)\Mto \Feas(x)$. If this holds for every $x\in X$, we simply say that $\Feas$ is \emph{weakly Mosco-continuous}.
\end{definition}

The two conditions defining weak Mosco-continuity are occasionally referred to as complete inner and weak outer semicontinuity. Observe moreover that, if $\Feas$ is weakly Mosco-continuous, then $\Feas(x)$ is weakly closed for all $x\in X$.

Note that we will give a sufficient condition for the weak Mosco-continuity of the feasible set mapping in terms of certain constraint qualifications in Section~\ref{Sec:KKT}.

\subsection{Pseudomonotone Operators}\label{Sec:PrelimsPseudo}

The purpose of this section is to discuss the concept of \emph{pseudomonotone operators} in the sense of Brezis \cite{Brezis1968}, see also \cite{Renardy2004,Zeidler1990}. Note that there is another property in the literature, introduced by Karamardian \cite{Karamardian1976}, which is also sometimes referred to as pseudomonotonicity. We stress that the two concepts are distinct and that, throughout the remainder of this paper, pseudomonotonicity will always refer to the property below.

\begin{definition}[Pseudomonotonicity]\label{Dfn:Pseudomonotone}
    We say that an operator $T:X\to X^*$ is \emph{pseudomonotone} if, for every sequence $\{x^k\}\subseteq X$ such that $x^k\wto x\in X$ and $\limsup_{k\to\infty} \langle T(x^k),x^k-x \rangle\le 0$, we have $\langle T(x),x-y \rangle\le \liminf_{k\to\infty} \langle T(x^k),x^k-y\rangle$ for all $y\in X$.
\end{definition}

Despite its peculiar appearance, pseudomonotone operators will prove immensely useful because they provide a unified approach to monotone and nonmonotone operators. It should be noted that, despite its name suggesting otherwise, pseudomonotonicity can also be viewed as a rather weak kind of continuity. In fact, many operators are pseudomonotone simply by virtue of satisfying certain continuity properties. Some important examples are summarized in the following lemma.

\begin{lemma}[Sufficient conditions for pseudomonotonicity]\label{Lem:PseudoSufficient}
    Let $X$ be a Banach space and $T,U:X\to X^*$ given operators. Then:
\begin{enumerate}[(a)]
\item If $T$ is monotone and continuous, then $T$ is pseudomonotone.
\item If $T$ is completely continuous, then $T$ is pseudomonotone.
\item If $T$ is continuous and $\dim (X)<+\infty$, then $T$ is pseudomonotone.
\item If $T$ and $U$ are pseudomonotone, then $T+U$ is pseudomonotone.
\end{enumerate}
\end{lemma}
\begin{proof}
    This is essentially \cite[Prop.~27.6]{Zeidler1990}. Note that \cite{Zeidler1990} defines pseudomonotonicity only on reflexive Banach spaces, but this property is not used in the proof of the result.
\end{proof}

It follows from the lemma above that, in particular, every continuous operator $T:X\to X^*$ on a finite-dimensional space $X$ is pseudomonotone, even if it is not monotone.

The following result shows that bounded pseudomonotone operators automatically enjoy some kind of continuity. Note that the result generalizes \cite[Prop.~27.7(b)]{Zeidler1990} since we do not assume the reflexivity of $X$.

\begin{lemma}\label{Lem:PseudoProperties}
    Let $F:X\to X^*$ be a bounded pseudomonotone operator. Then $F$ is \emph{demicontinuous}, i.e., it maps strongly convergent sequences to weak-$^*$ convergent sequences. In particular, if $\dim (X)<+\infty$, then $F$ is continuous.
\end{lemma}
\begin{proof}
    Let $\{x^k\}\subseteq X$ be a sequence with $x^k\to x$ for some $x\in X$. Observe that $\{F(x^k)\}$ is bounded in $X^*$ and hence
\begin{equation*}
    \bigl| \bigl\langle F(x^k), x^k-x \bigr\rangle \bigr|
    \le \|F(x^k)\|_{X^*} \|x^k-x\|_X \to 0.
\end{equation*}
    Thus, by pseudomonotonicity, we obtain
\begin{equation}\label{Eq:LemPseudoProperties1}
    \langle F(x),x-y \rangle \le \liminf_{k\to\infty}
    \big\langle F(x^k),x^k-y \big\rangle
    = \liminf_{k\to\infty} \big\langle F(x^k),x-y \big\rangle
\end{equation}
    for all $y\in X$, where we used the boundedness of $\{F(x^k)\}$ and the fact that $x^k\to x$. Inserting $\tilde{y}:=2 x-y$ for an arbitrary $y\in X$, we also obtain
\begin{equation}\label{Eq:LemPseudoProperties2}
    \langle F(x),y-x \rangle = \langle F(x),x-\tilde{y} \rangle \le
    \liminf_{k\to\infty} \big\langle F(x^k),x^k-\tilde{y} \big\rangle =
    \liminf_{k\to\infty} \big\langle F(x^k),y-x \big\rangle,
\end{equation}
    where the last equality uses the fact that $\{F(x^k)\}$ is bounded and that $x^k-\tilde{y}=y-x+o(1)$. Putting \eqref{Eq:LemPseudoProperties1} and \eqref{Eq:LemPseudoProperties2} together, it follows that $\langle F(x^k),x-y \rangle \to \langle F(x),x-y \rangle$ for all $y\in X$. This implies $F(x^k)\wto^* F(x)$, and the proof is done.
\end{proof}

In the context of QVIs, the pseudomonotonicity of the mapping $F$ plays a key role since it ensures, together with the weak Mosco-continuity of $\Feas$, that weak limit points of sequences of approximate solutions of the QVI are exact solutions.

\begin{proposition}\label{Prop:PseudomonStability}
    Let $F$ be a bounded pseudomonotone operator and let $\Feas$ be weakly Mosco-continuous. Assume that $\{x^k\}\subseteq X$ converges weakly to $\bar{x}$, that $\bar{x}\in\Feas(\bar{x})$, and that there are null sequences $\{\delta_k\},\{\varepsilon_k\}\subseteq \R$ (possibly negative) such that
\begin{equation}\label{Eq:PropPseudomonStability1}
    \langle F(x^k),y-x^k\rangle\ge \delta_k+\varepsilon_k \|y-x^k\|_X
    \quad\forall y\in \Feas(x^k)
\end{equation}
    for all $k$. Then $\bar{x}$ is a solution of the QVI.
\end{proposition}
\begin{proof}
    By Mosco-continuity, there is a sequence $\bar{x}^k\in \Feas(x^k)$ such that $\bar{x}^k\to \bar{x}$. Inserting $\bar{x}^k$ into \eqref{Eq:PropPseudomonStability1} yields $\liminf_{k\to\infty}\langle F(x^k),\bar{x}^k-x^k \rangle\ge 0$ and, since $\{F(x^k)\}$ is bounded, $\liminf_{k\to\infty}\langle F(x^k),\bar{x}-x^k\rangle\ge 0$. The pseudomonotonicity of $F$ therefore implies that
\begin{equation}\label{Eq:PropPseudomonStability2}
    \langle F(\bar{x}),y-\bar{x}\rangle\ge \limsup_{k\to\infty}
    \big\langle F(x^k),y-x^k \big\rangle \quad\text{for all }y\in X.
\end{equation}
    To show that $\bar{x}$ solves the QVI, let $y\in \Feas(\bar{x})$. Using the Mosco-continuity of $\Feas$, we obtain a sequence $y^k\in\Feas(x^k)$ such that $y^k\to y$. By \eqref{Eq:PropPseudomonStability1}, we have $\liminf_{k\to\infty}\langle F(x^k),y^k-x^k\rangle\ge 0$, hence $\liminf_{k\to\infty}\langle F(x^k),y-x^k \rangle\ge 0$, and \eqref{Eq:PropPseudomonStability2} implies that $\langle F(\bar{x}),y-\bar{x}\rangle\ge 0$.
\end{proof}

The above result will play a key role in our subsequent analysis. Note that the assumption \eqref{Eq:PropPseudomonStability1} can be relaxed; in fact, we only need the right-hand side to converge to zero whenever $y^k\in\Feas(x^k)$ and $\{y^k\}$ remains bounded. However, for our purposes, the formulation in \eqref{Eq:PropPseudomonStability1} is sufficient.

Let us also remark that, if $\Feas(x)\equiv \Feas$ is constant, then the Mosco-continuity in Proposition~\ref{Prop:PseudomonStability} is satisfied trivially and we obtain a stability result under pseudomonotonicity alone. In this case, it is easy to see that the boundedness of $F$ can be omitted.

\subsection{Convexity of the Feasible Set Mapping}\label{Sec:PrelimsConvexity}

As mentioned in the introduction, an important distinction in the context of QVIs is whether the feasible sets $\Feas(x)$, $x\in X$, are convex or not. This question is particularly important if the feasible set has the form \eqref{Eq:FeasibleSet}, since we need to clarify which requirements on the mapping $G$ are sufficient for the convexity of $\Feas(x)$. Ideally, these conditions should be easy and also yield some useful analytical properties.

Assume for the moment that $\ConY$ is a closed convex \emph{cone}. Then $\ConY$ induces an order relation, $y\le_{\ConY} z$ if and only if $z-y\in\ConY$, and $\ConY$ itself can be regarded as the nonnegative cone with respect to $\le_K$. Hence, it is natural to assume that the mapping $G$ is concave with respect to this ordering.

If $\ConY$ is not a cone, then the appropriate order relation turns out to be induced by the recession cone $\ConY_{\infty}:=\{ y\in\SpaceY: y+\ConY\subseteq \ConY \}$. Note that $\ConY_{\infty}$ is always a nonempty closed convex cone \cite{Bonnans2000} and therefore induces an order relation as outlined above.

\begin{definition}[$\ConY_{\infty}$-concavity]\label{Dfn:KConcave}
    We say that $G$ is \emph{$\ConY_{\infty}$-concave with respect to $y$} if the mapping $G(x,\cdot)$ is concave with respect to the order relation induced by $\ConY_{\infty}$, or equivalently
    \begin{equation*}
        G( x,(1-\alpha)y_1+\alpha y_2 )-(1-\alpha)G(x,y_1)-\alpha G(x,y_2)\in \ConY_{\infty}
    \end{equation*}
    for all $x,y_1,y_2\in X$ and $\alpha\in [0,1]$.
\end{definition}

Some important consequences of $\ConY_{\infty}$-concavity are formulated in the following lemma. A proof can be found in \cite[Lem.~2.1]{Kanzow2017c} for the case where $\SpaceY$ is a Hilbert space, and the general case is completely analogous.

\begin{lemma}[Properties of $\ConY_{\infty}$-concavity]\label{Lem:GeneralizedConvexity}
    Let $G$ be $\ConY_{\infty}$-concave with respect to $y$, and let $x\in X$. Then \textnormal{(i)} the function $y\mapsto \langle \lambda,G(x,y)\rangle$ is convex for all $\lambda\in\ConY_{\infty}^{\circ}$, \textnormal{(ii)} the function $y\mapsto d_{\ConY}(G(x,y))$ is convex, and \textnormal{(iii)} the set $\Feas(x)$ is convex.
\end{lemma}

\section{Regularity and Optimality Conditions}\label{Sec:KKT}

The purpose of this section is to provide a formal approach to first-order optimality conditions involving Lagrange multipliers for the QVI. As commonly done in the context of QVIs, we say that a point $x\in X$ is \emph{feasible} if $x\in\Feas (x)$.

The key observation for the first-order optimality conditions is the following: a point $\bar{x}$ solves the QVI if and only if $\bar{d}:=0$ minimizes the function $d\mapsto \langle F(\bar{x}),d\rangle$ over $d\in\TngCone{\Feas(\bar{x})}{\bar{x}}$. Under a suitable constraint qualification (see below), one can show that $\TngCone{\Feas(\bar{x})}{\bar{x}}=\{ d\in \TngCone{\ConX}{\bar{x}}: D_y G(\bar{x},\bar{x})d\in\TngCone{\ConY}{G(\bar{x},\bar{x})} \}$. Hence, this problem reduces to
\begin{equation}\label{Eq:QVItransformed}
    \min{} \langle F(\bar{x}),d\rangle \quad\text{s.t.}\quad d\in \TngCone{\ConX}{\bar{x}},~
    D_y G(\bar{x},\bar{x})d\in \TngCone{\ConY}{G(\bar{x},\bar{x})}.
\end{equation}
The KKT system of the QVI is now obtained by applying the standard theory of KKT conditions to this optimization problem in $\bar{d}=0$. This motivates the definition
\begin{equation}\label{Eq:L}
    \Lag:X\times \SpaceY^* \to X^*, \quad
    \Lag(x,\lambda):=F(x)+D_y G(x,x)^* \lambda,
\end{equation}
as the Lagrange function of the QVI. The corresponding first-order optimality conditions are given as follows.

\begin{definition}[KKT Conditions]\label{Dfn:KKT}
    A tuple $(\bar{x},\bar{\lambda})\in X\times \SpaceY^*$ is a \emph{KKT point} of \eqref{Eq:QVI}, \eqref{Eq:FeasibleSet}, if
\begin{equation}\label{Eq:KKT}
    -\Lag(\bar{x},\bar{\lambda})\in \NorCone{\ConX}{\bar{x}} \quad\text{and}\quad
    \bar{\lambda}\in \NorCone{\ConY}{G(\bar{x},\bar{x})}.
\end{equation}
    We call $\bar{x}$ a \emph{stationary point} if $(\bar{x},\bar{\lambda})$ is a KKT point for some multiplier $\bar{\lambda}\in\SpaceY^*$, and denote by $\Lambda(\bar{x})\subseteq \SpaceY^*$ the set of such multipliers.
\end{definition}

As mentioned before, a constraint qualification is necessary to obtain the assertion that every solution of the QVI admits a Lagrange multiplier. To this end, we apply the standard Robinson constraint qualification to \eqref{Eq:QVItransformed}.

\begin{definition}[Robinson constraint qualification]\label{Dfn:RobinsonCQ}
    Let $x\in X$ be an arbitrary point. We say that $x$ satisfies
\begin{enumerate}[(i)]
\item the \emph{extended Robinson constraint qualification (ERCQ)} if
\begin{equation}\label{Eq:RobinsonCQ}
    0\in \operatorname{int} \bigl[ G(x,x)+D_y G(x,x)(\ConX-x)-\ConY \bigr].
    \vspace{-\parskip}
\end{equation}
\item the \emph{Robinson constraint qualification (RCQ)} if $x$ is feasible and satisfies \eqref{Eq:RobinsonCQ}.
\end{enumerate}
\end{definition}

For nonlinear programming-type constraints, RCQ reduces to the Mangasarian--Fromovitz constraint qualification (MFCQ), see \cite[p.\,71]{Bonnans2000}. The extended RCQ is similar to the so-called \emph{extended MFCQ} which generalizes MFCQ to (possibly) infeasible points.

The connection between the QVI and its KKT conditions is much stronger than it is for optimization problems. The reason behind this is that, in a way, the QVI itself is already a problem formulation tailored towards ``first-order optimality''.

\begin{proposition}\label{Prop:KKT}
    If $(\bar{x},\bar{\lambda})$ is a KKT point, then $\bar{x}$ is a solution of the QVI. Conversely, if $\bar{x}$ is a solution of the QVI and RCQ holds in $\bar{x}$, then there exists a multiplier $\bar{\lambda}\in\SpaceY^*$ such that $(\bar{x},\bar{\lambda})$ is a KKT point, and the corresponding multiplier set is bounded.
\end{proposition}
\begin{proof}
    Throughout the proof, let $\mathsf{G}(x,y):=(y,G(x,y))$ and $\mathsf{\ConY}:=\ConX\times\ConY$, so that
\begin{equation}\label{Eq:PropKKT1}
    \Feas(x)=\{ y\in\ConX: G(x,y)\in\ConY \} =
    \{ y\in X: \mathsf{G}(x,y)\in\mathsf{\ConY} \}
    \quad\text{for all }x\in X.
\end{equation}
    An easy calculation shows that the KKT conditions remain invariant under this reformulation of the constraint system. By \cite[Lem.~2.100]{Bonnans2000}, the same holds for RCQ.

    Now, let $(\bar{x},\bar{\lambda})$ be a KKT point of the QVI. Using \eqref{Eq:PropKKT1} and \cite[Thm.~2.4]{Kanzow2017c}, it follows that $\bar{x}$ is a solution of the variational inequality $\bar{x}\in S$, $\langle F(\bar{x}),d\rangle\ge 0$ for all $d\in \TngCone{S}{\bar{x}}$, where $S:=\Feas(\bar{x})$ is considered fixed. But this means that $\bar{x}$ solves the QVI.
    
    Conversely, if $\bar{x}$ solves the QVI and RCQ holds in $\bar{x}$, then \eqref{Eq:PropKKT1} and \cite[Cor.~2.91]{Bonnans2000} imply that $\TngCone{\Feas(\bar{x})}{\bar{x}}=\{ d\in \TngCone{\ConX}{\bar{x}}: D_y G(\bar{x},\bar{x})d\in\TngCone{\ConY}{G(\bar{x},\bar{x})} \}$, so that $\bar{x}$ is a solution of \eqref{Eq:QVItransformed}. The (ordinary) Robinson constraint qualification \cite[Def.~2.86]{Bonnans2000} for this problem takes on the form
\begin{equation*}
    0\in\operatorname{int}[ D_y G(\bar{x},\bar{x})\TngCone{\ConX}{\bar{x}}-\TngCone{\ConY}{G(\bar{x},\bar{x})}].
\end{equation*}
    Since $\ConX-\bar{x}\subseteq\TngCone{\ConX}{\bar{x}}$ and $\ConY-G(\bar{x},\bar{x})\subseteq \TngCone{\ConY}{G(\bar{x},\bar{x})}$, this condition is implied by \eqref{Eq:RobinsonCQ}. The result now follows by applying a standard KKT theorem to \eqref{Eq:QVItransformed}, such as \cite[Thm.~3.9]{Bonnans2000}.
\end{proof}

We now turn to another consequence of RCQ which is a certain metric regularity of the feasible set(s). This property can also be used to deduce the weak Mosco-continuity of $\Feas$ (see Corollary~\ref{Cor:MoscoRCQ}). The main idea is that, given some reference point $\bar{x}$, we can view the parameter $x$ in the feasible set mapping $\Feas(x)$ as a perturbation parameter and use the following result from perturbation theory.

\begin{lemma}\label{Lem:MetricRegularityRCQ}
    Let $\bar{x}\in X$ and $\bar{y}\in\Feas(\bar{x})$. Assume that $G$ and $D_y G$ are continuous on $U\times X$, where $U$ is the space $X$ equipped with an arbitrary topology, and that
\begin{equation*}
    0\in \operatorname{int} \bigl[ G(\bar{x},\bar{y})+D_y G(\bar{x},\bar{y})(\ConX-\bar{y})-\ConY \bigr].
\end{equation*}
    Then there are $c>0$ and a neighborhood $N$ of $(\bar{x},\bar{y})$ in $U\times X$ such that $\operatorname{dist}(y,\Feas(x))\le c \operatorname{dist}(G(x,y),\ConY)$ for all $(x,y)\in N$ with $y\in \ConX$.
\end{lemma}
\begin{proof}
    Similar to above, let $\mathsf{G}:U\times X\to X\times \SpaceY$ be the mapping $\mathsf{G}(x,y):=(y,G(x,y))$, and define the set $\mathsf{\ConY}:=\ConX\times\ConY$. Observe that $\mathsf{G}$ is Fr\'echet-differentiable with respect to $y$, that both $\mathsf{G}$ and $D_y\mathsf{G}$ are continuous on $U\times X$, and that $\Feas(x)=\{ y\in X: \mathsf{G}(x,y)\in \mathsf{\ConY} \}$ for all $x\in X$. By \cite[Lem.~2.100]{Bonnans2000}, we have $0\in \operatorname{int} [ \mathsf{G}(\bar{x},\bar{y})+D_y \mathsf{G}(\bar{x},\bar{y}) X-\mathsf{\ConY} ]$. Hence, by \cite[Thm.~2.87]{Bonnans2000}, there are $c>0$ and a neighborhood $N$ of $(\bar{x},\bar{y})$ in $U\times X$ such that
\begin{equation*}
    \dist{y,\Feas(x)}\le c \dist{\mathsf{G}(x,y),\mathsf{\ConY}}
\end{equation*}
    for all $(x,y)\in N$. Clearly, if $y\in C$, then $\operatorname{dist}(\mathsf{G}(x,y),\mathsf{\ConY})=\operatorname{dist}(G(x,y),\ConY)$.
\end{proof}

As mentioned before, we can use Lemma~\ref{Lem:MetricRegularityRCQ} to prove the weak Mosco-continuity of $\Feas$. To this end, we only need to apply the lemma in the special case where $U$ is the space $X$ equipped with the weak sequential topology.

\begin{corollary}\label{Cor:MoscoRCQ}
    Let $\Feas(x)=\{ y\in\ConX: G(x,y)\in\ConY \}$ and let $\bar{x}$ be a feasible point. Assume that $G$ is $\ConY_{\infty}$-concave with respect to $y$, that $G$ and $D_y G$ satisfy the continuity property
    \begin{equation*}
        x^k\wto x,\quad y^k\to y \quad\implies\quad G(x^k,y^k)\to G(x,y)
        ,\quad D_y G(x^k,y^k) \to D_y G(x,y)
    \end{equation*}
    for all $x,y\in X$, and that RCQ holds in $\bar{x}$. Then $\Feas$ is weakly Mosco-continuous in $\bar{x}$.
\end{corollary}
\begin{proof}
    Let $x^k\wto \bar{x}$ and $y^k\in \Feas(x^k)$, $y^k\wto \bar{y}$. Then $\{y^k\}\subseteq \ConX$, which implies $\bar{y}\in\ConX$. Moreover, $G(x^k,y^k)\in\ConY$ for all $k$, which implies $G(\bar{x},\bar{y})\in\ConY$ and $\bar{y}\in\Feas(\bar{x})$.
    
    For the inner semicontinuity, let $x^k\wto \bar{x}$ and $\bar{y}\in \Feas(\bar{x})$. As in the proof of Proposition~\ref{Prop:KKT}, we may assume that $\ConX=X$. By assumption, the constraint system $y\in\ConX$, $G(\bar{x},y)\in\ConY$ satisfies the (ordinary) RCQ in $\bar{x}$; thus, by convexity, this constraint system satisfies RCQ in every $y\in \Feas(\bar{x})$ (see, e.g., \cite[Theorems 2.83 and 2.104]{Bonnans2000}), in particular for $y:=\bar{y}$. Now, let $U$ denote the space $X$ equipped with the weak sequential topology. Then $G$ and $D_y G$ are continuous on $U\times X$. By Lemma~\ref{Lem:MetricRegularityRCQ}, there exists $c>0$ such that
\begin{equation*}
    \dist{\bar{y},\Feas(x^k)}\le c \dist{G(x^k,\bar{y}),\ConY}
\end{equation*}
    for $k\in\N$ sufficiently large. Since $G(x^k,\bar{y})\to G(\bar{x},\bar{y})\in\ConY$ by assumption, it follows that the right-hand side converges to zero as $k\to\infty$. Hence, we can choose points $y^k\in \Feas(x^k)$ with $\|y^k-\bar{y}\|_X \to 0$. This completes the proof.
\end{proof}

\section{An Existence Result for QVIs}\label{Sec:Existence}

The existence of solutions to QVIs is a rather delicate topic. Many results, especially for infinite-dimensional problems, either deal with specific problem settings \cite{Adly2010,Kunze2000} or consider general QVIs under rather long lists of assumptions, often including monotonicity \cite{Baiocchi1984,Flores2000,Mosco1976}. Two interesting exceptions are the papers \cite{Kim1988,Shih1985}, which deal with quite general classes of QVIs and prove existence results under suitable compactness and continuity assumptions. However, the results contained in these papers actually require the \emph{complete continuity} of the mapping $F$. This is a very restrictive assumption which cannot be expected to hold in many applications; for instance, it does not even hold if $X$ is a Hilbert space and $F(x)=x$. An analogous comment applies if $F$ involves additional summands, e.g., if $F$ arises from the derivative of an optimal control-type objective function with a Tikhonov regularization parameter.

In this paper, we pursue a different approach which is based on a combination of the weak Mosco-continuity from Section~\ref{Sec:PrelimsWeakTop} and the Brezis-type pseudomonotonicity from Section~\ref{Sec:PrelimsPseudo}. A particularly intuitive idea is given by Proposition~\ref{Prop:PseudomonStability}, which suggests that we can tackle the QVI by solving a sequence of approximating problems and then using a suitable limiting argument to obtain a solution of the problem of interest. For the precise implementation of this idea, we will need some auxiliary results.

The first result we need is a slight modification of an existence theorem of Brezis, Nirenberg, and Stampacchia for VI-type equilibrium problems, see \cite[Thm.~1]{Brezis1972}. Note that the result in \cite{Brezis1972} is formulated in a rather general setting and uses filters instead of sequences; however, when applied to the Banach space setting, one can dispense with filters by using, for instance, Day's lemma \cite[Lem.~2.8.5]{Megginson1998}. We shall not demonstrate the resulting proof here, mainly for the sake of brevity and since it is basically identical to that given in \cite{Brezis1972}. The interested reader will also find the complete proof in the dissertation \cite{Steck2018}.

\begin{proposition}[Brezis--Nirenberg--Stampacchia]\label{Prop:BrezisNirenbergStampacchia}
    Let $A\subseteq X$ be a nonempty, convex, weakly compact set, and $\Psi:A\times A\to\R$ a mapping such that
\begin{enumerate}[(i)]
\item $\Psi(x,x)\le 0$ for all $x\in A$,
\item for every $x\in A$, the function $\Psi(x,\cdot)$ is (quasi-)concave,
\item for every $y\in A$ and every finite-dimensional subspace $L$ of $X$, the function $\Psi(\cdot,y)$ is lower semicontinuous on $A\cap L$, and
\item whenever $x,y\in A$, $\{x^k\}\subseteq A$ converges weakly to $x$, and $\Psi(x^k,(1-t)x+t y)\le 0$ for all $t\in [0,1]$ and $k\in\N$, then $\Psi(x,y)\le 0$.
\end{enumerate}
    Then there exists $\hat{x}\in A$ such that $\Psi(\hat{x},y)\le 0$ for all $y\in A$.
\end{proposition}

We will mainly need Proposition~\ref{Prop:BrezisNirenbergStampacchia} to obtain the existence of solutions to the approximating problems in our existence result for QVIs. For this purpose, it will be useful to present a slightly more tangible corollary of the result.

\begin{corollary}\label{Cor:ExistenceBNS}
    Let $A\subseteq X$ be a nonempty, convex, weakly compact set, $F:X\to X^*$ a bounded pseudomonotone operator, and $\varphi:A^2\to\R$ a mapping such that
\begin{enumerate}[(i)]
\item $\varphi(x,x)=0$ for all $x\in A$,
\item for every $x\in A$, the function $\varphi(x,\cdot)$ is concave, and
\item for every $y\in A$, the function $\varphi(\cdot,y)$ is weakly sequentially lsc.
\end{enumerate}
    Then there exists $\hat{x}\in A$ such that $\langle F(\hat{x}),\hat{x}-y \rangle + \varphi (\hat{x},y)\le 0$ for all $y\in A$.
\end{corollary}
\begin{proof}
    We claim that the mapping $\Psi:A^2\to\R$, $\Psi(x,y):=\langle F(x),x-y \rangle+\varphi(x,y)$, satisfies the assumption of Proposition~\ref{Prop:BrezisNirenbergStampacchia}. Clearly, $\Psi(x,x)\le 0$ for every $x\in A$, and $\Psi$ is (quasi-)concave with respect to the second argument. Moreover, by the properties of pseudomonotone operators (\cref{Lem:PseudoProperties}), $\Psi$ is lower semicontinuous with respect to the first argument on $A\cap L$ for any finite-dimensional subspace $L$ of $X$. Finally, let $x,y\in A$, let $\{x^k\}\subseteq A$ be a sequence converging weakly to $x$, and assume that
\begin{equation}\label{Eq:CorExistenceBNS1}
    \Psi(x^k,(1-t)x+t y)\le 0 \quad \forall t\in [0,1],~\forall k\in\N.
\end{equation}
    We need to show that $\Psi(x,y)\le 0$. By \eqref{Eq:CorExistenceBNS1}, we have in particular that $\Psi(x^k,x)\le 0$ and $\Psi(x^k,y)\le 0$ for all $k$. The first of these conditions implies that
\begin{align*}
    0\ge \limsup_{k\to\infty} \Psi (x^k,x)
    & \ge \limsup_{k\to\infty} \big\langle F(x^k),x^k-x \big\rangle + \liminf_{k\to\infty} \varphi(x^k,x) \\
    & \ge \limsup_{k\to\infty} \big\langle F(x^k),x^k-x \bigr\rangle,
\end{align*}
where we used the weak sequential lower semicontinuity of $\varphi$ with respect to $x$ and the fact that $\varphi(x,x)=0$. Hence, by the pseudomonotonicity of $F$, we obtain
\begin{align*}
    \Psi(x,y) & =\langle F(x),x-y \rangle+\varphi(x,y) \\
    & \le \liminf_{k\to\infty} \big[ \bigl\langle F(x^k),x^k-y \bigr\rangle
    + \varphi(x^k,y) \big] = \liminf_{k\to\infty} \Psi(x^k,y) \le 0.
\end{align*}
Therefore, $\Psi$ satisfies all the requirements of Proposition~\ref{Prop:BrezisNirenbergStampacchia}, and the result follows.
\end{proof}

Apart from the above result, we will also need some information on the behavior of the ``parametric'' distance function $(x,y)\mapsto d_{\Feas(x)}(y)$. Here, the weak Mosco-continuity of $\Feas$ plays a key role and allows us to prove the following lemma.

\begin{lemma}\label{Lem:MoscoSemicontinuity}
    Let $\Feas:X\rightrightarrows X$ be weakly Mosco-continuous. Then, for every $y\in X$, the distance function $x\mapsto d_{\Feas(x)}(y)$ is weakly sequentially upper semicontinuous on $X$.
    
    If, in addition, there are nonempty subsets $A,B\subseteq X$ such that $\Feas(A)\subseteq B$ and $B$ is weakly compact, then the function $x\mapsto d_{\Feas(x)}(x)$ is weakly sequentially lsc on $A$.
\end{lemma}
\begin{proof}
    Let $y\in X$ be a fixed point and let $\{x^k\}\subseteq X$, $x^k\wto x\in X$. If $w\in \Feas(x)$ is an arbitrary point, then there is a sequence $w^k\in \Feas(x^k)$ such that $w^k\to w$. It follows that
\begin{equation*}
    \|y-w\|_X=\lim_{k\to\infty} \|y-w^k\|_X \ge
    \limsup_{k\to\infty} d_{\Feas(x^k)}(y).
\end{equation*}
    Since $w\in\Feas(x)$ was arbitrary, this implies that $d_{\Feas(x)}(y)\ge \limsup_{k\to\infty} d_{\Feas(x^k)}(y)$.
    
    We now prove the second assertion. Let $\{x^k\}\subseteq A$ be a sequence with $x^k\wto x\in A$. Without loss of generality, let $d_{\Feas(x^k)}(x^k)\to \liminf_{k\to\infty} d_{\Feas(x^k)}(x^k)$, and let $\Feas(x^k)$ be nonempty for all $k$. Choose points $w^k\in\Feas(x^k)$ such that $\|x^k-w^k\|_X\le d_{\Feas(x^k)}(x^k)+1/k$. By assumption, the sequence $\{w^k\}$ is contained in the weakly compact set $B$, and thus there is an index set $\Subseq\subseteq \N$ such that $w^k\wto_{\Subseq}w$ for some $w\in B$. Since $w^k\in\Feas(x^k)$ for all $k$, the weak Mosco-continuity of $\Feas$ implies $w\in\Feas(x)$. It follows that
\begin{equation*}
    d_{\Feas(x)}(x)\le \|x-w\|_X \le
    \liminf_{k\in\Subseq} \|x^k-w^k\|_X =
    \liminf_{k\to\infty} d_{\Feas(x^k)}(x^k).
\end{equation*}
    This completes the proof.
\end{proof}


We now turn to the main existence result for QVIs.

\begin{theorem}\label{Thm:Existence}
    Consider a QVI of the form \eqref{Eq:QVI}. Assume that
\begin{enumerate*}[(i)]
\item $F$ is bounded and pseudomonotone,
\item $\Feas$ is weakly Mosco-continuous, and
\item there is a nonempty, convex, weakly compact set $A\subseteq X$ such that, for all $x\in A$, $\Feas(x)$ is nonempty, closed, convex, and contained in $A$.
\end{enumerate*}
    Then the QVI admits a solution $\bar{x}\in A$.
\end{theorem}
\begin{proof}
    For $k\in\N$, let $\Psi_k:A^2\to\R$ be the bifunction
\begin{equation*}
    \Psi_k(x,y):=\langle F(x),x-y \rangle+
    k \bigl[ d_{\Feas(x)}(x)-d_{\Feas(x)}(y)  \bigr].
\end{equation*}
    By \cref{Lem:MoscoSemicontinuity,Cor:ExistenceBNS}, there exist points $x^k\in A$ such that $\Psi_k(x^k,y)\le 0$ for all $y\in A$. Since $A$ is weakly compact, the sequence $\{x^k\}$ admits a weak limit point $\bar{x}\in A$. Moreover, by assumption, there are points $y^k\in\Feas(x^k)\subseteq A$ for all $k$. For these points, we obtain
\begin{equation*}
    0\ge \Psi_k(x^k,y^k)=\bigl\langle F(x^k),x^k-y^k \bigr\rangle + k d_{\Feas(x^k)}(x^k).
\end{equation*}
    By the boundedness of $A$ and $F$, the first term is bounded. Hence, dividing by $k$, we obtain $d_{\Feas(x^k)}(x^k)\to 0$, thus $d_{\Feas(\bar{x})}(\bar{x})=0$ by \cref{Lem:MoscoSemicontinuity}, and hence $\bar{x}\in\Feas(\bar{x})$.
    
    Finally, we claim that $\bar{x}$ solves the QVI. Observe that, for all $k$ and $y\in\Feas(x^k)$,
\begin{equation*}
    0\ge \Psi_k(x^k,y)= \bigl\langle F(x^k),x^k-y \bigr\rangle +k d_{\Feas(x^k)}(x^k) \ge \bigl\langle F(x^k),x^k-y \bigr\rangle.
\end{equation*}
    Thus, by \cref{Prop:PseudomonStability}, it follows that $\bar{x}$ is a solution of the QVI.
\end{proof}

The applicability of the above theorem depends most crucially on the weak Mosco-continuity of $\Feas$ and the existence of the weakly compact set $A$. It should be possible to modify the theorem by requiring some form of coercivity instead, but this is outside the scope of the present paper.

\begin{example}\label{Ex:GradientConstraints}
This example is based on \cite{Kunze2000}. Let $\Omega\subseteq \R^d$, $d\ge 2$, be a bounded domain, let $X:=H_0^1(\Omega)$, and consider the QVI given by the functions
\begin{equation*}
    F(u):=-\Delta u-f, \qquad \Feas(u):=\{ v\in H_0^1(\Omega): \|\nabla v\|\le \Psi(u) \},
\end{equation*}
where $\|\cdot\|$ is the Euclidean norm, $f\in H^{-1}(\Omega)$, and $\Psi: H_0^1(\Omega)\to L^{\infty}(\Omega)$ is completely continuous. Observe that $F$ is pseudomonotone by Lemma~\ref{Lem:PseudoSufficient}(i). Assume now that $c_1\le\Psi(u)\le c_2$ for all $u$ and some $c_1,c_2>0$. Then $\Feas$ is weakly Mosco-continuous by \cite[Lem.~1]{Kunze2000}. Moreover, $0\in\Feas(u)$ for all $u\in H_0^1(\Omega)$, and the Poincar\'e inequality implies that there is an $R>0$ such that $\Feas(u)\subseteq B_R^X$ for all $u\in H_0^1(\Omega)$. We conclude that all the requirements of Theorem~\ref{Thm:Existence} are satisfied; hence, the QVI admits a solution.
\end{example}

%

\section{The Augmented Lagrangian Method}\label{Sec:Method}

We now present the augmented Lagrangian method for the QVI \eqref{Eq:QVI}. The main approach is to penalize the function $G$ and therefore reduce the QVI to a sequence of standard VIs. Throughout the remainder of the paper, we assume that $i:\SpaceY\hookrightarrow\SpaceH$ densely for some real Hilbert space $\SpaceH$, and that $\ConH$ is a closed convex subset of $\SpaceH$ with $i^{-1}(\ConH)=\ConY$.

Consider the augmented Lagrangian $\Lag_{\rho}:X\times \SpaceH\to X^*$ given by
\begin{equation}\label{Eq:AL}
   \Lag_{\rho}(x,\lambda):=F(x)+\rho D_y G(x,x)^* \mleft[ G(x,x)+\frac{\lambda}{\rho}
   -P_{\ConH}\mleft( G(x,x)+\frac{\lambda}{\rho}\mright) \mright].
\end{equation}
Note that, if $\ConH$ is a cone, then we can simplify the above formula to $\Lag_{\rho}(x,\lambda)=F(x)+D_y G(x,x)^* P_{\ConH^{\circ}}(\lambda+\rho G(x,x))$ by using Moreau's decomposition \cite{Bauschke2011,Moreau1962}.

For the construction of our algorithm, we will need a means of controlling the penalty parameters. To this end, we define the utility function
\begin{equation}\label{Eq:V}
   V(x,\lambda,\rho):=
   \left\|G(x,x)-P_{\ConH}\mleft(G(x,x)+\frac{\lambda}{\rho}\mright)\right\|_{\SpaceH}.
\end{equation}
The function $V$ is a composite measure of feasibility and complementarity; it arises from an inherent slack variable transformation which is often used to define the augmented Lagrangian for inequality or cone constraints.

\begin{algorithm}[Augmented Lagrangian method]\label{Alg:ALM}
Let $(x^0,\lambda^0)\in X\times \SpaceH$, $\rho_0>0$, $\gamma>1$, $\tau\in(0,1)$, let $B\subseteq \SpaceH$ be a bounded set, and set $k:=0$.
\begin{enumerate}[topsep=1ex,parsep=0ex,leftmargin=*,label=\textbf{Step~\arabic*.}]
    \item If $(x^k,\lambda^k)$ satisfies a suitable termination 
        criterion: STOP.
    \item Choose $\BddMul^k\in B$ and compute an inexact solution (see below)
        $x^{k+1}$ of the VI
\begin{equation}\label{Eq:PenOpt}
    x\in \ConX, \quad \bigl\langle \Lag_{\rho_k}(x,\BddMul^k),y-x\bigr\rangle
    \ge 0 \quad \forall y\in \ConX.
\end{equation}
    \item Update the vector of multipliers to
        \begin{equation}\label{Eq:MultUpdate}
            \lambda^{k+1}:=\rho_k \mleft[ G(x^{k+1},x^{k+1})+\frac{\BddMul^k}{\rho_k}
            -P_{\ConH}\mleft(G(x^{k+1},x^{k+1})+\frac{\BddMul^k}{\rho_k}\mright) \mright].
        \end{equation}
    \item If $k=0$ or
        \begin{equation}\label{Eq:RhoTest}
            V(x^{k+1},\BddMul^k,\rho_k)\le \tau V(x^k,\BddMul^{k-1},\rho_{k-1})
        \end{equation}
        holds, set $\rho_{k+1}:=\rho_k$; otherwise, set $\rho_{k+1}:=\gamma\rho_{k}$.
    \item Set $k\leftarrow k+1$ and go to Step~1.
\end{enumerate}
\end{algorithm}

Let us make some simple observations. First, regardless of the primal iterates $\{x^k\}$, the multipliers $\{\lambda^k\}$ always lie in the polar cone $\ConH_{\infty}^{\circ}$ by \cite[Lem.~2.2]{Kanzow2017c}. Moreover, if $\ConH$ is a cone, then the well-known Moreau decomposition for closed convex cones implies that $\lambda^{k+1}=P_{\ConH^{\circ}}(\BddMul^k+\rho_k G(x^{k+1},x^{k+1}))$.

Secondly, we note that Algorithm \ref{Alg:ALM} uses a safeguarded multiplier sequence $\{\BddMul^k\}$ in certain places where classical augmented Lagrangian methods use the sequence $\{\lambda^k\}$. This bounding scheme goes back to \cite{Andreani2007,Pang2005} and is crucial to establishing good global convergence results for the method \cite{Andreani2007,Birgin2014,Kanzow2016b}. In practice, one usually tries to keep $\BddMul^k$ as ``close'' as possible to $\lambda^k$, e.g., by defining $\BddMul^k:=P_B(\lambda^k)$, where $B$ (the bounded set from the algorithm) is chosen suitably to allow cheap projections.

Our final observation concerns the definition of $\lambda^{k+1}$. Regardless of the manner in which $x^{k+1}$ is computed (exactly or inexactly), we always have the equality
\begin{equation}\label{Eq:AL_L}
    \Lag_{\rho_k}(x^{k+1},\BddMul^k)=\Lag(x^{k+1},\lambda^{k+1})
    \quad\text{for all }k\in\N,
\end{equation}
which follows directly from the definition of $\Lag_{\rho}$ and the multiplier updating scheme \eqref{Eq:MultUpdate}. This equality is the main motivation for the definition of $\lambda^{k+1}$.

We now prove a lemma which essentially asserts some kind of ``approximate normality'' of $\lambda^k$ and $G(x^k,x^k)$. Recall that the KKT conditions of the QVI require that $\bar{\lambda}\in\NorCone{\ConY}{G(\bar{x},\bar{x})}$. However, we have to take into account that $G(x^k,x^k)$ is not necessarily an element of $\ConY$.

\begin{lemma}\label{Lem:ApproxNormality}
    There is a null sequence $\{r_k\}\subseteq [0,\infty)$ such that $\bigl(\lambda^k,y-G(x^k,x^k)\bigr)\le r_k$ for all $y\in\ConH$ and $k\in\N$.
\end{lemma}
\begin{proof}
    Let $y\in \ConH$ and define the sequence $s^{k+1}:=P_{\ConH}(G(x^{k+1},x^{k+1})+\BddMul^k/\rho_k)$. Then $s^{k+1}\in \ConH$ and $\lambda^{k+1}\in\NorCone{\ConH}{s^{k+1}}$ by \cite[Prop.~6.46]{Bauschke2011}. Moreover, we have
    \begin{equation}\label{Eq:LemApproxNormality1}
        G(x^{k+1},x^{k+1})=\frac{\lambda^{k+1}-\BddMul^k}{\rho_k}+s^{k+1}.
    \end{equation}
    This yields
\begin{align}
    \mleft(\lambda^{k+1},y-G(x^{k+1},x^{k+1})\mright) & =
    \mleft(\lambda^{k+1},y-\frac{1}{\rho_k}(\lambda^{k+1}-\BddMul^k)-s^{k+1}\mright) \notag \\
    & \le \frac{1}{\rho_k} \Bigl[\bigl(\lambda^{k+1},\BddMul^k\bigr)-\|\lambda^{k+1}\|_{\SpaceH}^2\Bigr]
    \label{Eq:LemApproxNormality2},
\end{align}
    where we used $\lambda^{k+1}\in\NorCone{\ConH}{s^{k+1}}$ for the last inequality. We now show that the sequence $\{r_k\}$ given by the right-hand side satisfies $\limsup_{k\to\infty} r_k\le 0$. This yields the desired result (by replacing $r_k$ with $\max\{0,r_k\}$). If $\{\rho_k\}$ is bounded, then \eqref{Eq:RhoTest} and \eqref{Eq:LemApproxNormality1} imply $\|\lambda^{k+1}-\BddMul^k\|_{\SpaceH}/\rho_k\to 0$ and therefore $\|\lambda^{k+1}-\BddMul^k\|_{\SpaceH}\to 0$. This yields the boundedness of $\{\lambda^{k+1}\}$ in $\SpaceH$ as well as $(\lambda^{k+1},\BddMul^k)-\|\lambda^{k+1}\|_{\SpaceH}^2=(\lambda^{k+1},\BddMul^k-\lambda^{k+1})\to 0$. Hence, $r_k\to 0$. We now assume that $\rho_k\to\infty$. Note that \eqref{Eq:LemApproxNormality2} is a quadratic function in $\lambda$. A simple calculation therefore shows that $r_k\le \|\BddMul^k\|_{\SpaceH}^2/(4\rho_k)$ and, hence, $\limsup_{k\to\infty} r_k\le 0$.
\end{proof}

Let us point out that the inequality in the above lemma is uniform since the sequence $\{r_k\}$ does not depend on the point $y\in\ConH$. Moreover, we remark that the proof uses only the definition of $\lambda^{k+1}$ and does not make any assumption on the sequence $\{x^k\}$.

\section{Global Convergence for Convex Constraints}\label{Sec:ConvexConv}

In this section, we analyze the convergence properties of Algorithm~\ref{Alg:ALM} for QVIs where the set $\Feas(x)$ is convex for all $x$. In the situation where $\Feas(x)=\{ y\in C: G(x,y)\in\ConY \}$ as in \eqref{Eq:FeasibleSet}, the natural analytic notion of convexity is the $\ConH_{\infty}$-concavity of $G$ with respect to $y$, see Section~\ref{Sec:Prelims}. To reflect this, we make the following set of assumptions.

\begin{assumption}\label{Asm:Convex}
    We assume that
\begin{enumerate*}[(i)]
\item $F$ is bounded and pseudomonotone,
\item $\Feas$ is weakly Mosco-continuous,
\item $G$ is $\ConH_{\infty}$-concave with respect to $y$,
\item $d_{\ConH}\circ G$ is weakly sequentially lsc, and
\item $x^{k+1}\in\ConX$ and $\varepsilon^{k+1}-\Lag_{\rho_k}(x^{k+1},\BddMul^k)\in \NorCone{\ConX}{x^{k+1}}$ for all $k$, where $\varepsilon^k\to 0$.
\end{enumerate*}
\end{assumption}

Note that (i) and (ii) were also used in Theorem~\ref{Thm:Existence}. Moreover, the assumption on the sequence $\{x^k\}$ is just an inexact version of the VI subproblem \eqref{Eq:PenOpt}.

We continue by proving the feasibility and optimality of weak limit points of the sequence $\{x^k\}$. The first result in this direction is the following.

\begin{lemma}\label{Lem:Feasibility_Cvx}
    Let Assumption~\ref{Asm:Convex} hold and let $\bar{x}$ be a weak limit point of $\{x^k\}$.
    If $\Feas(\bar{x})$ is nonempty, then $\bar{x}$ is feasible.
\end{lemma}
\begin{proof}
    Let $\Subseq\subseteq\N$ be an index set such that $x^{k+1}\wto_{\Subseq}\bar{x}$. Observe first that $\bar{x}\in\ConX$ since $\ConX$ is closed and convex, hence weakly sequentially closed. It therefore remains to show that $G(\bar{x},\bar{x})\in\ConY$. If $\{\rho_k\}$ remains bounded, then the penalty updating scheme \eqref{Eq:RhoTest} implies
\begin{equation*}
    d_{\ConH}(G(x^{k+1},x^{k+1}))\le \mleft\| G(x^{k+1},x^{k+1})-
    P_{\ConH}\mleft( G(x^{k+1},x^{k+1})+\frac{\BddMul^k}{\rho_k}\mright)
    \mright\|_{\SpaceH} \to 0.
\end{equation*}
    Since $d_{\ConH}\circ G$ is weakly sequentially lsc, we obtain $d_{\ConH}(G(\bar{x},\bar{x}))=0$ and thus $G(\bar{x},\bar{x})\in\ConY$. Assume now that $\rho_k\to\infty$ and that $G(\bar{x},\bar{x})\notin \ConY$, or equivalently $d_{\ConH}^2(G(\bar{x},\bar{x}))>0$. Since $\Feas(\bar{x})$ is nonempty, we can choose an $y\in \Feas(\bar{x})$, and by inner semicontinuity there exists a sequence $y^{k+1}\in \Feas(x^{k+1})$ such that $y^{k+1}\to_{\Subseq} y$. Now, let $h_k(x,y):=d_{\ConH}(G(x,y)+\BddMul^k/\rho_k)$. Observe that $h_k$ is convex in $y$ by Lemma~\ref{Lem:GeneralizedConvexity}, and that $h_k^2$ is continuously differentiable in $y$ by \cite[Cor.~12.30]{Bauschke2011}. Moreover, since the distance function is nonexpansive, it follows that $h_k(x^{k+1},y^{k+1})\le \|\BddMul^k\|_H/\rho_k\to 0$. Thus, using \cref{Asm:Convex}(iv), we obtain
\begin{equation*}
    \liminf_{k\to\infty} \bigl[ h_k(x^{k+1},x^{k+1})-h_k(x^{k+1},y^{k+1}) \bigr]
    =\liminf_{k\to\infty} h_k(x^{k+1},x^{k+1})
    \ge d_{\ConH}(G(\bar{x},\bar{x})) >0.
\end{equation*}
    Hence, there is a constant $c_1>0$ such that $h_k^2(x^{k+1},x^{k+1})-h_k^2(x^{k+1},y^{k+1})\ge c_1$ for all $k\in\Subseq$ sufficiently large. The convexity of $h_k$ (and of $h_k^2$) now yields
\begin{equation*}
    \bigl\langle D_y (h_k^2)(x^{k+1},x^{k+1}),y^{k+1}-x^{k+1}\bigr\rangle\le
    h_k^2(x^{k+1},y^{k+1})-h_k^2(x^{k+1},x^{k+1}) \le -c_1.
\end{equation*}
    Furthermore, by (i), there is a constant $c_2\in\R$ such that $\langle F(x^{k+1}),x^{k+1}-y^{k+1}\rangle\ge c_2$ for all $k\in\Subseq$. Now, let $\{\varepsilon^k\}$ be the sequence from Assumption~\ref{Asm:Convex}. Observe that $\Lag_{\rho_k}(x^{k+1},\BddMul^k)=F(x^{k+1})+(\rho_k/2) D_y (h_k^2)(x^{k+1},x^{k+1})$. Therefore,
\begin{equation*}
    \bigl\langle\varepsilon^{k+1},y^{k+1}-x^{k+1}\bigr\rangle
    \le \bigl\langle \Lag_{\rho_k}(x^{k+1},\BddMul^k),y^{k+1}-x^{k+1}\bigr\rangle
    \le -\frac{\rho_k c_1}{2}-c_2\to -\infty,
\end{equation*}
    which contradicts $\varepsilon^{k+1}\to 0$.
\end{proof}

Since the augmented Lagrangian method is, at its heart, a penalty-type algorithm, the attainment of feasibility is paramount to the success of the algorithm. The above lemma gives us some information in this direction since it guarantees that the weak limit point $\bar{x}$ is ``as feasible as possible'' in the sense that it is feasible if (and only if) $\Feas(\bar{x})$ is nonempty. In many particular examples of QVIs (e.g., for moving-set problems), we know a priori that $\Feas(x)$ is nonempty for all $x\in X$, and this directly yields the feasibility of $\bar{x}$.

\begin{theorem}\label{Thm:Optimality_Cvx}
    Let Assumption~\ref{Asm:Convex} hold and let $\bar{x}$ be a weak limit point of $\{x^k\}$. If $\Feas(\bar{x})$ is nonempty, then $\bar{x}$ is feasible and a solution of the QVI.
\end{theorem}
\begin{proof}
    The feasibility follows from Lemma~\ref{Lem:Feasibility_Cvx}. For the optimality part, we apply Proposition~\ref{Prop:PseudomonStability}. To this end, let $y^k\in \Feas(x^k)$, let $\{\varepsilon^k\}$ be the sequence from Assumption~\ref{Asm:Convex}, and recall that $\Lag_{\rho_k}(x^{k+1},\BddMul^k)=\Lag(x^{k+1},\lambda^{k+1})$. Then
\begin{align*}
    \bigl\langle \varepsilon^{k+1},y^{k+1}-x^{k+1}\bigr\rangle & \le
    \bigl\langle F(x^{k+1})+ D_y G(x^{k+1},x^{k+1})^*\lambda^{k+1},y^{k+1}-x^{k+1}\bigr\rangle \\
    & = \bigl\langle F(x^{k+1}),y^{k+1}-x^{k+1} \bigr\rangle + \bigl\langle\lambda^{k+1},
    D_y G(x^{k+1},x^{k+1})(y^{k+1}-x^{k+1})\bigr\rangle \\
    & \le \bigl\langle F(x^{k+1}),y^{k+1}-x^{k+1}\bigr\rangle+\bigl\langle\lambda^{k+1},
    G(x^{k+1},y^{k+1})-G(x^{k+1},x^{k+1})\bigr\rangle,
\end{align*}
    where we used the convexity of $y\mapsto \langle \lambda^{k+1},G(x^{k+1},y)\rangle$ which follows from Lemma~\ref{Lem:GeneralizedConvexity}. Since $G(x^{k+1},y^{k+1})\in \ConH$, the last term is bounded from above by $r_{k+1}$, where $\{r_k\}$ is the null sequence from Lemma~\ref{Lem:ApproxNormality}. The result therefore follows from Proposition~\ref{Prop:PseudomonStability}.
\end{proof}

The above theorem guarantees the optimality of any weak limit point of the sequence generated by Algorithm~\ref{Alg:ALM}. Despite this, it should be pointed out that the result is purely ``primal'' in the sense that no assertions are made for the multiplier sequence. We will investigate the dual (or, more precisely, primal-dual) behavior of the augmented Lagrangian method in more detail in Section~\ref{Sec:GeneralConv}.

If the mapping $F$ is strongly monotone, then we obtain strong convergence of the iterates.

\begin{corollary}\label{Cor:StrMonotone}
    Let Assumption~\ref{Asm:Convex} hold and assume that there is a $c>0$ such that
\begin{equation*}
    \langle F(x)-F(y),x-y\rangle\ge c \|x-y\|_X^2 \quad\text{for all }x,y\in X.
\end{equation*}
    If $x^k\wto_{\Subseq}\bar{x}$ for some subset $\Subseq\subseteq\N$ and $\bar{x}$ is a solution of the QVI, then $x^k\to_{\Subseq}\bar{x}$.
\end{corollary}
\begin{proof}
    By the weak Mosco-continuity of $\Feas$, there is a sequence $\bar{x}^k\in\Feas (x^k)$ such that $\bar{x}^k\to \bar{x}$. Recalling the proof of Theorem~\ref{Thm:Optimality_Cvx}, we have $\liminf_{k\in\Subseq}\langle F(x^k),\bar{x}^k-x^k\rangle\ge 0$ and therefore $\liminf_{k\in\Subseq}\langle F(x^k),\bar{x}-x^k\rangle\ge 0$. The strong monotonicity of $F$ yields
\begin{equation*}
    c\|x^k-\bar{x}\|_X^2\le \bigl\langle F(x^k)-F(\bar{x}),x^k-\bar{x}\bigr\rangle
    =\bigl\langle F(x^k),x^k-\bar{x}\bigr\rangle -
    \bigl\langle F(\bar{x}),x^k-\bar{x} \bigr\rangle.
\end{equation*}
    But the $\limsup$ of the first term is less than or equal to zero, and the second term converges to zero since $x^k\wto_{\Subseq}\bar{x}$. Hence, $\|x^k-\bar{x}\|_X\to 0$, and the proof is complete.
\end{proof}

%

\section{Primal-Dual Convergence Analysis}\label{Sec:GeneralConv}

The purpose of this section is to establish a formal connection between the primal-dual sequence $\{(x^k,\lambda^k)\}$ generated by the augmented Lagrangian method and the KKT conditions of the QVI. The main motivation for this is the relationship \eqref{Eq:AL_L} between the augmented Lagrangian and the ordinary Lagrangian, and the approximate normality of $\lambda^k$ and $G(x^k,x^k)$ from Lemma~\ref{Lem:ApproxNormality}. These two components essentially constitute an asymptotic version of the KKT conditions and suggest that a careful analysis of the primal-dual sequence $\{(x^k,\lambda^k)\}$ could lead to suitable optimality assertions.

\begin{assumption}\label{Asm:GeneralConv}
    We assume that
\begin{enumerate*}[(i)]
\item $F$ bounded and pseudomonotone,
\item the mappings $G$ and $D_y G$ are completely continuous, and
\item $x^{k+1}\in\ConX$ and $\varepsilon^{k+1}-\Lag_{\rho_k}(x^{k+1},\BddMul^k)\in \NorCone{\ConX}{x^{k+1}}$ for all $k$, where $\varepsilon^k\to 0$.
\end{enumerate*}
\end{assumption}

The above assumption is certainly natural in the sense that $x^{k+1}$ is an approximate solution of the corresponding VI subproblem, and that the degree of inexactness vanishes as $k\to\infty$. However, we obviously need to clarify whether this assumption can be satisfied in practice. To this end, we obtain the following result.

\begin{lemma}\label{Lem:SubproblemSolutions}
    Let Assumption~\ref{Asm:GeneralConv} \textnormal{(i)-(iii)} hold and assume that $\ConX$ is weakly compact. Then the VI subproblems \eqref{Eq:PenOpt} admit solutions for every $k\in\N$.
\end{lemma}
\begin{proof}
    Observe that, for all $k\in\N$, the function $x\mapsto \Lag_{\rho_k}(x,\BddMul^k)$ is pseudomonotone by Lemma~\ref{Lem:PseudoSufficient}. Hence, by \cref{Cor:ExistenceBNS}, the corresponding VIs admit solutions in $\ConX$.
\end{proof}

Our next result deals with the feasibility of the iterates. As observed in the previous section, the attainment of feasibility is crucial to the success of penalty-type methods such as the augmented Lagrangian method. This aspect is even more important for QVIs due to the inherently difficult structure of the constraints.

\begin{lemma}\label{Lem:Feasibility}
    Let Assumption~\ref{Asm:GeneralConv} hold and let $\bar{x}$ be a weak limit point of $\{x^k\}$. Then $\bar{x}\in\ConX$ and $-D_y (d_{\ConH}^2\circ G)(\bar{x},\bar{x}) \in\NorCone{\ConX}{\bar{x}}$. If $\bar{x}$ satisfies ERCQ, then $\bar{x}$ is feasible.
\end{lemma}
\begin{proof}
    Clearly, $\bar{x}\in \ConX$ since $\ConX$ is weakly sequentially closed. If $\{\rho_k\}$ is bounded, then \eqref{Eq:RhoTest} implies $d_{\ConH}(G(x^{k+1},x^{k+1}))\to 0$, which yields $G(\bar{x},\bar{x})\in \ConH$ since $G$ is completely continuous. Hence, there is nothing to prove. Now, let $\rho_k\to\infty$, let $x^{k+1}\wto_{\Subseq}\bar{x}$ on some subset $\Subseq\subseteq\N$, and let $\{\varepsilon^k\}\subseteq X^*$ be the sequence from Assumption \ref{Asm:GeneralConv}. Then
    \begin{equation*}
        \varepsilon^{k+1}-F(x^{k+1})-D_y G(x^{k+1},x^{k+1})^* \lambda^{k+1}
        \in \NorCone{\ConX}{x^{k+1}}
    \end{equation*}
    for all $k\in\N$. We now divide this inclusion by $\rho_k$, use the definition of $\lambda^{k+1}$ and the fact that $\NorCone{\ConX}{x^{k+1}}$ is a cone. It follows that
\begin{align*}
    -D_y G(x^{k+1},x^{k+1})^* \biggl[ G(x^{k+1},x^{k+1})+\frac{\BddMul^k}{\rho_k}
    & -P_{\ConH} \mleft( G(x^{k+1},x^{k+1})+\frac{\BddMul^k}{\rho_k} \mright)\biggr] \\
    & +\frac{\varepsilon^{k+1}-F(x^{k+1})}{\rho_k}\in\NorCone{\ConX}{x^{k+1}}.
\end{align*}
    Note that $G(x^{k+1},x^{k+1})\to G(\bar{x},\bar{x})$ and $D_y G(x^{k+1},x^{k+1})\to D_y G(\bar{x},\bar{x})$ by complete continuity. Thus, taking the limit $k\to_{\Subseq}\infty$ and using a standard closedness property of the normal cone mapping, we obtain
\begin{equation}\label{Eq:ThmFeasibility1}
    -D_y G(\bar{x},\bar{x})^* [ G(\bar{x},\bar{x})-
    P_{\ConH}(G(\bar{x},\bar{x})) ]\in \NorCone{\ConX}{\bar{x}},
\end{equation}
    which is the first claim. Assume now that ERCQ holds in $\bar{x}$, and let $r>0$ be such that $B_r^{\SpaceY}\subseteq G(\bar{x},\bar{x})+D_y G(\bar{x},\bar{x}) (\ConX-\bar{x})-\ConY$. Then, for any $y\in B_r^{\SpaceY}$, there are $z\in \ConX$ and $w\in\ConY$ such that $y = G(\bar{x},\bar{x})+D_y G(\bar{x},\bar{x})(z-\bar{x})-w$. In particular, we have
\begin{align*}
    \langle G(\bar{x},\bar{x})-P_{\ConH}(G(\bar{x},\bar{x})),y\rangle ={} &
    \bigl\langle D_y G(\bar{x},\bar{x})^* \bigl[G(\bar{x},\bar{x})-P_{\ConH}(G(\bar{x},\bar{x}))\bigr]
    ,z-\bar{x}\bigr\rangle \\
    & + \langle G(\bar{x},\bar{x})-P_{\ConH}(G(\bar{x},\bar{x})),G(\bar{x},\bar{x})-w\rangle.
\end{align*}
    The first term is nonnegative by \eqref{Eq:ThmFeasibility1}, and so is the second term by standard projection inequalities. Hence, $\langle G(\bar{x},\bar{x})-P_{\ConH}(G(\bar{x},\bar{x})),y\rangle\ge 0$ for all $y\in B_r^{\SpaceY}$, which implies $\langle G(\bar{x},\bar{x})-P_{\ConH}(G(\bar{x},\bar{x})),y\rangle=0$ for all $y\in B_r^{\SpaceY}$ and, since $\SpaceY$ is dense in $\SpaceH$, it follows that $G(\bar{x},\bar{x})-P_{\ConH}(G(\bar{x},\bar{x}))=0$. This completes the proof.
\end{proof}

\noindent
The above is our main feasibility result for this section. Note that the assertion $-D_y (d_{\ConH}^2\circ G)(\bar{x},\bar{x})\in\NorCone{\ConX}{\bar{x}}$ has a rather natural interpretation: the function $d_{\ConH}^2\circ G$ is a measure of infeasibility, and the lemma above states that any weak limit point $\bar{x}$ of $\{x^k\}$ has a (partial) minimization property for this function on the set $\ConX$.

The observation that an extended form of RCQ yields the actual feasibility of the point $\bar{x}$ is motivated by similar arguments for penalty-type methods in finite dimensions. In fact, a common condition used in the convergence theory of such methods is the extended MFCQ (see Section~\ref{Sec:Prelims}), which is a generalization of the ordinary MFCQ to points which are not necessarily feasible.

Having dealt with the feasibility aspect, we now turn to the main primal-dual convergence result.

\begin{theorem}\label{Thm:Optimality}
    Let Assumption~\ref{Asm:GeneralConv} hold, let $x^k\wto_{\Subseq}\bar{x}$ on some subset $\Subseq\subseteq\N$, and assume that $\bar{x}$ satisfies ERCQ. Then $\bar{x}$ is feasible, the sequence $\{\lambda^k\}_{k\in\Subseq}$ is bounded in $\SpaceY^*$, and each of its weak-$^*$ accumulation points is a Lagrange multiplier corresponding to $\bar{x}$.
\end{theorem}
\begin{proof}
    The feasibility of $\bar{x}$ follows from Lemma~\ref{Lem:Feasibility}. Observe furthermore that $G(x^k,x^k)\to_{\Subseq} G(\bar{x},\bar{x})$ and $D_y G(x^k,x^k)\to_{\Subseq} D_y G(\bar{x},\bar{x})$. If $\{r_k\}$ and $\{\varepsilon^k\}$ are the sequences from Lemma~\ref{Lem:ApproxNormality} and Assumption~\ref{Asm:GeneralConv} respectively, then we have
\begin{equation}\label{Eq:ThmOptimality1}
    \varepsilon^k-\Lag(x^k,\lambda^k)\in \NorCone{\ConX}{x^k}
    \quad\text{and}\quad
    \bigl( \lambda^k,y-G(x^k,x^k) \bigr)\le r_k ~\forall y\in\ConH
\end{equation}
    for all $k\ge 1$. To prove the boundedness of $\{\lambda^k\}_{k\in\Subseq}$, we now proceed as follows. Applying the generalized open mapping theorem \cite[Thm.~2.70]{Bonnans2000} to the multifunction $\Psi(u):=G(\bar{x},\bar{x})+D_y G(\bar{x},\bar{x})u-\ConY$ on the domain $\ConX-\bar{x}$, we see that there is an $r>0$ such that
\begin{equation*}
    B_r^{\SpaceY} \subseteq G(\bar{x},\bar{x}) + D_y G(\bar{x},\bar{x})
    \mleft[(\ConX-\bar{x})\cap B_1^X\mright]-\ConY.
\end{equation*}
    By the definition of the dual norm, we can choose a sequence $\{y^k\}\subseteq \SpaceY$ of unit vectors such that $\langle\lambda^k,y^k\rangle\ge\frac{1}{2}\|\lambda^k\|_{
    \SpaceY^*}$. For every $k$, we can write
\begin{equation*}
    -r y^k = G(\bar{x},\bar{x})+D_y G(\bar{x},\bar{x})(v^k-\bar{x})-z^k
\end{equation*}
    with $\{v^k\}\subseteq \ConX$ bounded and $\{z^k\}\subseteq\ConY$. It follows that
    $r y^k = z^k-G(x^k,x^k)-D_y G(x^k,x^k)(v^k-x^k)+\delta^k$ with $\delta^k\to 0$
    as $k\to_{\Subseq}\infty$. Assume now that $k$ is large enough so that
    $\|\delta^k\|_{\SpaceY}\le r/4$. Then, by \eqref{Eq:ThmOptimality1},
\begin{align*}
    \frac{r}{2} \|\lambda^k\|_{\SpaceY^*} \le \bigl\langle\lambda^k,r y^k\bigr\rangle
    & \le \bigl\langle \lambda^k,z^k-G(x^k,x^k)\bigr\rangle-
    \bigl\langle \lambda^k,D_y G(x^k,x^k)(v^k-x^k) \bigr\rangle
    +\frac{r}{4} \|\lambda^k\|_{\SpaceY^*} \\
    & \le \bigl\langle\lambda^k,z^k-G(x^k,x^k)\bigr\rangle+\bigl\langle F(x^k)-\varepsilon^k,v^k-x^k\bigr\rangle
    +\frac{r}{4}\|\lambda^k\|_{\SpaceY^*}.
\end{align*}
    By \eqref{Eq:ThmOptimality1}, the first two terms are bounded from above (for $k\in\Subseq$) by some constant $c>0$. Reordering the above inequality yields
    $\frac{r}{4}\|\lambda^k\|_{\SpaceY^*}\le c$, and the result follows.
    
    Finally, let us show that every weak-$^*$ limit point of $\{\lambda^k\}_{k\in\Subseq}$ is a Lagrange multiplier. Without loss of generality, we assume that $\lambda^k\wto_{\Subseq}^* \bar{\lambda}$ for some $\bar{\lambda}\in\SpaceY^*$ on the same subset $\Subseq\subseteq\N$ where $\{x^k\}$ converges weakly to $\bar{x}$. By \eqref{Eq:ThmOptimality1} and standard properties of the normal cone, we have $\bar{\lambda}\in \NorCone{\ConY}{G(\bar{x},\bar{x})}$. Now, let $y\in \ConX$. Then, by \eqref{Eq:ThmOptimality1},
\begin{equation}\label{Eq:ThmOptimality2}
    \bigl\langle \varepsilon^k,y-x^k\bigr\rangle\le \bigl\langle F(x^k),y-x^k\bigr\rangle+\bigl\langle\lambda^k,D_y G(x^k,x^k) (y-x^k)\bigr\rangle.
\end{equation}
    By complete continuity, we have $D_y G(x^k,x^k)\to_{\Subseq} D_y G(\bar{x},\bar{x})$ and $D_y G(\bar{x},\bar{x})(y-x^k)\to_{\Subseq} D_y G(\bar{x},\bar{x})(y-\bar{x})$, see \cite[Thm.~1.5.1]{Emelyanov2007}. Hence, $D_y G(x^k,x^k)(y-x^k)\to_{\Subseq} D_y G(\bar{x},\bar{x})(y-\bar{x})$. We now argue similarly to Proposition~\ref{Prop:PseudomonStability}: inserting $y:=\bar{x}$ into \eqref{Eq:ThmOptimality2} yields $\liminf_{k\in\Subseq}\langle F(x^k),\bar{x}-x^k\rangle \ge 0$. Using the pseudomonotonicity of $F$ and the fact that $\lambda^k\wto_{\Subseq}^*\bar{\lambda}$, we obtain that
\begin{equation*}
    \langle F(\bar{x}),y-\bar{x} \rangle +
    \langle \bar{\lambda},D_y G(\bar{x},\bar{x})(y-\bar{x}) \rangle \ge
    \limsup_{k\in\Subseq} \bigl[ \bigl\langle F(x^k),y-x^k \rangle+
    \bigl\langle \lambda^k,D_y G(x^k,x^k)(y-x^k) \bigr\rangle \bigr].
\end{equation*}
    Since $y\in\ConX$ was arbitrary, this means that $-\Lag(\bar{x},\bar{\lambda})\in\NorCone{\ConX}{\bar{x}}$. Hence, $(\bar{x},\bar{\lambda})$ is a KKT point of the QVI.
\end{proof}

\section{Applications and Numerical Results}\label{Sec:Applic}

In this section, we present some numerical applications of our theoretical framework. Let us start by observing that our algorithm generalizes various methods from finite dimensions, e.g., for QVIs \cite{Kanzow2016,Kanzow2017a,Pang2005} or generalized Nash equilibrium problems \cite{Kanzow2016a}. Hence, any of the applications in those papers remain valid for the present one.

However, the much more interesting case is that of QVIs which are fundamentally infinite-dimensional. In the following, we will present three such examples, discuss their theoretical background, and then apply the augmented Lagrangian method to discretized versions of the problems. The implementation was done in MATLAB and uses the algorithmic parameters
\begin{equation*}
    \lambda^0:=0, \quad \rho_0:=1, \quad \tau=0.1, \quad \gamma:=10,
    \quad B:=[-10^6,10^6],
\end{equation*}
where $B$ is understood in the pointwise sense. The VI subproblems arising in the algorithm are solved by a semismooth Newton-type method. The outer and inner iterations are terminated when the residual of the corresponding first-order system drops below a certain threshold; the corresponding tolerances are chosen as $10^{-4}$ $(10^{-6})$ for the outer (inner) iterations in Examples~1 and 3, and $10^{-6}$ $(10^{-8})$ in Example~2.

Since our examples are defined in function spaces, we will typically use the notation $(u,v)$ instead of $(x,y)$ for the variable pairs in the space $X^2$. This should be rather clear from the context and hopefully does not introduce any confusion. Moreover, for a function space $\SpaceY$, we will denote by $\SpaceY_+$ the nonnegative cone in $\SpaceY$.

\subsection{An Implicit Signorini Problem}\label{Sec:ApplicSignorini}

The application presented here is an implicit Signorini-type problem \cite{Bensoussan1984,Mosco1976}. Let $\Omega\subseteq\R^2$ be a bounded domain with smooth boundary $\Gamma$, and let $X$ denote the space
\begin{equation*}
    X:=\{ u\in H^1(\Omega): \Delta u\in L^2(\Omega) \},
    \quad \text{with norm }
    \|u\|_X:=\|u\|_{H^1(\Omega)}+\|\Delta u\|_{L^2(\Omega)}.
\end{equation*}
Recall that the trace operator $\tau$ maps $H^1(\Omega)$ into $H^{1/2}(\Gamma)$, that $H^{1/2}(\Gamma)^*=:H^{-1/2}(\Gamma)$ \cite{Adams2003}, and that the normal derivative $\partial_n:X\to H^{-1/2}(\Gamma)$ is well-defined and continuous \cite{Tartar2007}. For fixed elements $h_0,\phi\in H^{1/2}(\Gamma)$ with $\phi\ge 0$, consider the set-valued mapping
\begin{equation*}
    \Feas(u):=\{ v\in X: \tau v\ge h(u)\text{ on }\Gamma \}, \quad
    h(u):=h_0-\langle\phi,\partial_n u\rangle,
\end{equation*}
where $h:H^1(\Omega)\to H^{1/2}(\Gamma)$ and the duality pairing is understood between $H^{1/2}(\Gamma)$ and $H^{-1/2}(\Gamma)$. The problem in question now is the QVI
\begin{equation*}
    u\in \Feas(u), \quad \langle A u-f,v-u\rangle\ge 0 \quad \forall v\in \Feas(u),
\end{equation*}
where $A:X\to X^*$ is a monotone differential operator and $f\in H^{-1}(\Omega)$. This problem can be cast into our general framework by choosing
\begin{equation*}
\begin{gathered}
    X=\{ u\in H^1(\Omega): \Delta u\in L^2(\Omega) \}, \quad \ConX:=X,
    \quad F(u):=A u-f,\\
    \SpaceY:=H^{1/2}(\Gamma), \quad G(u,v):=\tau v-h(u), \quad
    \ConY:=\SpaceY_+.
\end{gathered}
\end{equation*}
We claim that the implicit Signorini problem satisfies Assumption~\ref{Asm:Convex}. The pseudomonotonicity of $F$ follows from Lemma~\ref{Lem:PseudoSufficient}, and the mapping $G$ is weakly sequentially continuous. Moreover, the Mosco-continuity of $\Feas$ is rather straightforward to prove since $u\mapsto \langle\phi,\partial_n u\rangle$ maps into a one-dimensional subspace of $\SpaceY$.

It follows from the theory in Section~\ref{Sec:ConvexConv} that every weak limit point of the sequence $\{u^k\}$ generated by Algorithm~\ref{Alg:ALM} is a solution of the QVI. (Note that $\Feas(u)$ is nonempty for all $u\in X$, as was assumed in Lemma~\ref{Lem:Feasibility_Cvx}.) Observe furthermore that, since $\ConX=X$, the sequences $\{u^k\}$ and $\{\lambda^k\}$ generated by the algorithm satisfy
\begin{equation*}
    0 \leftarrow \bigl\langle F(u^k),h \bigr\rangle+\bigl\langle \tau^* \lambda^k,h\bigr\rangle
    = \bigl\langle F(u^k),h \bigr\rangle + \bigl\langle \lambda^k,\tau h\bigr\rangle
\end{equation*}
for all $h\in X$. Since the image $\tau (X)$ of the trace operator contains $H^{3/2}(\Gamma)$ \cite{Adams2003}, it follows that a subsequence of $\{\lambda^k\}$ converges weak-$^*$ in $H^{3/2}(\Gamma)^*$.

\begin{figure}\centering
    \subfloat[Solution $\bar{u}$]{\includegraphics[scale=0.5]{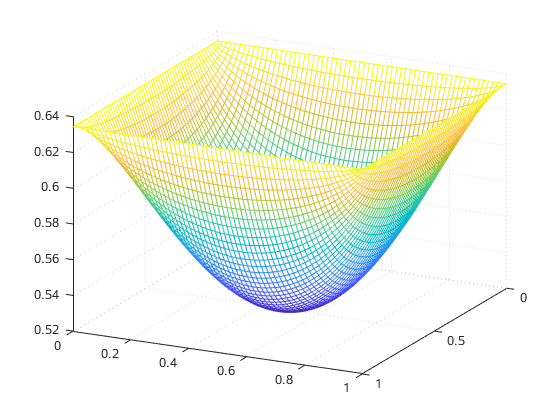}}
    \subfloat[Constraint value $G(\bar{u},\bar{u})$]{\includegraphics[scale=0.5]{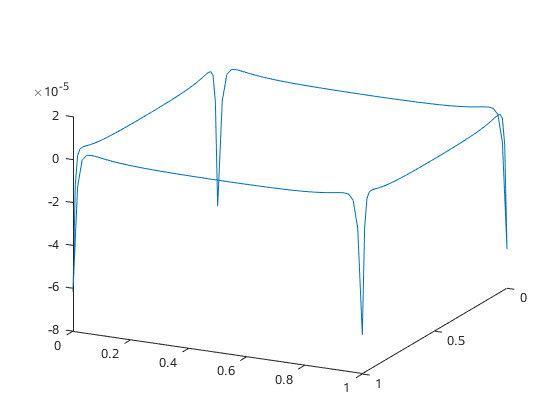}}
    \caption{Numerical results for the implicit Signorini problem with $n=64$.}
    \label{Fig:Signorini}
\end{figure}

We now present some numerical results for the domain $\Omega:=(0,1)^2$ and the differential operator $A u:=u-\Delta u$. For the implementation of the method, we set $\SpaceH:=L^2(\Gamma)$, $\ConH:=\SpaceH_+$, and discretize the domain $\Omega$ by means of a uniform grid with $n\in\N$ points per row or column (including boundary points), i.e., $n^2$ points in total. The remaining problem parameters are given by $f\equiv -1$ and $\phi=h_0\equiv 1$.
\begin{center}
\begin{tabular}{r|ccccc}
    $n$ & $16$ & $32$ & $64$ & $128$ & $256$ \\ \hline
	outer it. & 9 & 9 & 9 & 10 & 10 \\
	inner it. & 42 & 41 & 42 & 47 & 52 \\
	$\rho_{\max}$ & $10^4$ & $10^4$ & $10^5$ & $10^5$ & $10^5$
\end{tabular}
\end{center}
We observe that the method scales rather well with increasing dimension $n$. In particular, the outer iteration numbers and final penalty parameters remain nearly constant, and the increase in terms of inner iteration numbers is very moderate.

\subsection{Parametric Gradient Constraints}

This example is based on the theoretical framework in \cite{Hintermueller2012,Kunze2000} and represents a QVI with pointwise gradient-type constraints. Note that we have already given a brief discussion of such examples in Section~\ref{Sec:Existence}. Let $d,p\ge 2$ and $X:=W_0^{1,p}(\Omega)$ for some bounded domain $\Omega\subseteq \R^d$. Consider the set-valued mapping
\begin{equation}\label{Eq:ParametricGradientConstraint}
    \Feas(u):=\{ v\in X: \|\nabla v\|\le \Psi (u) \}
\end{equation}
where $\|\cdot\|$ is the Euclidean norm in $\R^d$ and $\Psi: X\to L^{\infty}(\Omega)$, as well as the resulting QVI
\begin{equation*}
    u\in \Feas(u), \quad \langle -\Delta_p u-f,v-u\rangle\ge 0 \quad \forall v\in \Feas(u),
\end{equation*}
where $f\in X^*$ and $\Delta_p:X\to X^*$ is the $p$-Laplacian defined by
\begin{equation*}
    \langle \Delta_p u,v\rangle := - \int_{\Omega} \|\nabla u(x)\|^{p-2} \nabla u(x)^T \nabla v(x)
    \,\textnormal{d}x.
\end{equation*}
Hence, we have $F(u):=-\Delta_p u-f$. Observe that $F$ is monotone, bounded, and continuous \cite{Hintermueller2012}, hence pseudomonotone by Lemma~\ref{Lem:PseudoSufficient}(i). Assume now that $\Psi$ is completely continuous and satisfies $\Psi u\ge c_1$ for all $u$ and some $c_1>0$. Then $\Feas$ is weakly Mosco-continuous by \cite[Lem.~1]{Kunze2000}.

When applying the augmented Lagrangian method to the above problem, a significant challenge lies in the analytical formulation of the feasible set. Observe that the original formulation in \eqref{Eq:ParametricGradientConstraint} is nonsmooth. This issue is probably not critical since the nonsmoothness is a rather ``mild'' one, but nevertheless an alternative formulation is necessary to formally apply our algorithmic framework. To this end, we first discretize the problem and then reformulate the (finite-dimensional) gradient constraint as $G(u,v)\le 0$ with $G(u,v)=\|\nabla v\|^2-\Psi(u)^2$.

For the discretized problems, both Assumption~\ref{Asm:Convex} and \ref{Asm:GeneralConv} are satisfied. Moreover, the feasible sets $\Feas(u)$ are nonempty for all $u\in X$. Hence, it follows from Theorem~\ref{Thm:Optimality_Cvx} that every limit point $\bar{u}$ of the (finite-dimensional) sequence $\{u^k\}$ is a solution of the QVI. For the boundedness of the multiplier sequence, it remains to verify that RCQ holds in $\bar{u}$. In fact, RCQ (which, in this case, is just MFCQ) holds everywhere, since the point zero is a Slater point of the mapping $v\mapsto G(u,v)$ for all $u\in X$.

\begin{figure}\centering
    \subfloat[Solution $\bar{u}$]{\includegraphics[scale=0.5]{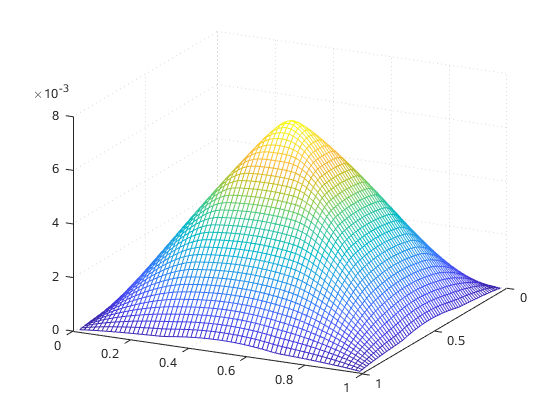}}
    \subfloat[Constraint value $G(\bar{u},\bar{u})$]{\includegraphics[scale=0.5]{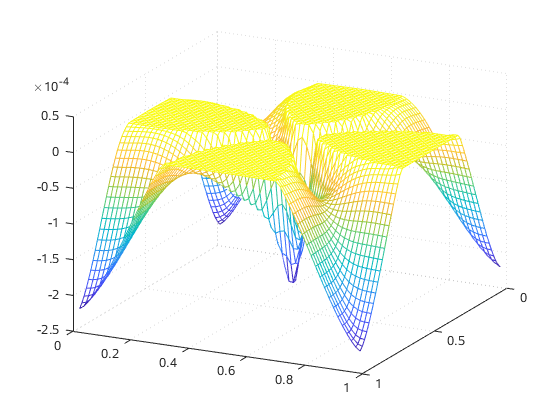}}
    \caption{Numerical results for the gradient-constrained QVI with $n=64$.}
    \label{Fig:GradientConstrained}
\end{figure}

As a numerical application, we consider a slightly modified version of \cite[Example~3]{Hintermueller2012} on the unit square $\Omega:=(0,1)^2$. Note that the original example in this reference is actually solved by the solution of the $p$-Laplace equation $-\Delta_p u-f=0$, $u\in W_0^{1,p}(\Omega)$. Hence, we have modified the example to use the same function $f$ as in \cite{Hintermueller2012} but with the constraint function replaced by $\Psi(u):=0.01+2|\int_{\Omega} u(x) \,\textnormal{d}x|$.

Similarly to \cite{Hintermueller2012}, we discretize the problem with $n\in\N$ interior points per row or column, and use backward differences to approximate the gradient and $p$-Laplace operators.
\begin{center}
\begin{tabular}{r|ccccc}
    $n$ & $16$ & $32$ & $64$ & $128$ & $256$ \\ \hline
	outer it. & 8 & 7 & 7 & 7 & 7 \\
	inner it. & 63 & 52 & 63 & 73 & 103 \\
	$\rho_{\max}$ & $10^4$ & $10^4$ & $10^4$ & $10^4$ & $10^5$
\end{tabular}
\end{center}
As before, the method scales rather well with increasing dimension $n$, and the outer iteration numbers and final penalty parameters remain nearly constant.

\subsection{Generalized Nash Equilibrium Problems}

The example presented in this section is a generalized Nash equilibrium problem (GNEP) arising from optimal control. Let $N\in\N$ be the number of players, each in control of a variable $u^{\nu}\in X_{\nu}$ for some real Banach space $X_{\nu}$. We set $X:=X_1\times\ldots\times X_N$ and write $u=(u^{\nu},u^{-\nu})$, $X=X_{\nu}\times X_{-\nu}$, to emphasize the role of player $\nu$ in the vector $u$ or the space $X$. The GNEP consists of $N$ optimization problems of the form
\begin{equation*}
    \min_{u^{\nu}\in X_{\nu}}\ J_{\nu}(u) \quad\text{s.t.}\quad u^{\nu}\in \Feas_{\nu}(u^{-\nu}),
\end{equation*}
where $J_{\nu}:X\to\R$ is the objective function of player $\nu$ and $\Feas_{\nu}:X_{-\nu}\rightrightarrows X_{\nu}$ the corresponding feasible set. Throughout this section, we consider a special type of GNEP arising from optimal control. To this end, let $\Omega\subseteq\R^d$, $d\in\{2,3\}$, be a bounded Lipschitz domain, and set $X:=L^2(\Omega)^N$. Furthermore, let $S: L^2(\Omega)\to H_0^1(\Omega)\cap C(\bar{\Omega})$ be the solution operator of the standard Poisson equation \cite{Troeltzsch2010}, let $\psi\in C(\bar{\Omega})$, $f,y_{\textnormal{d}}^{\nu}\in L^2(\Omega)$, $\alpha_{\nu}>0$ for all $\nu$, and $y(u):=S(u^1+\ldots+u^N+f)$. Consider the objective functions
\begin{equation*}
    J_{\nu}(u):=\frac{1}{2}\|y(u)-y_{\textnormal{d}}^{\nu}\|_{L^2(\Omega)}^2
    +\frac{\alpha_{\nu}}{2}\|u^{\nu}\|_{L^2(\Omega)}^2,
\end{equation*}
as well as the feasible set mappings
\begin{equation*}
    \Feas_{\nu}(u^{-\nu}):=\{ u^{\nu}\in U_{\text{ad}}^{\nu}: y(u^{\nu},u^{-\nu}) \ge \psi \},
\end{equation*}
where $U_{\text{ad}}^{\nu}\subseteq L^2(\Omega)$ is the set of admissible controls for player $\nu$.

The existence of solutions of the GNEP can be shown as in \cite{Hintermueller2015}. The GNEP can also be rewritten as a QVI by defining $F(u):=(D_{u^{\nu}}J_{\nu})_{\nu=1}^N$ and $\Feas(u):=\prod_{\nu=1}^N\Feas_{\nu}(u^{-\nu})$. The mapping $F$ is continuous and strongly monotone \cite{Kanzow2017b}, hence pseudomonotone. Observe that $\Feas$ can be written in analytic form by setting $\ConX:=\prod_{\nu=1}^N U_{\text{ad}}^{\nu}$, $\SpaceY:=C(\bar{\Omega})^N$, $\ConY:=\SpaceY_+$, and
\begin{equation*}
    G(u,v):=\bigl( y(u^{\nu},v^{-\nu})-\psi \bigr)_{\nu=1}^N.
\end{equation*}
Using the structure of $G$, it is now easy to see that the QVI-tailored RCQ from Section~\ref{Sec:KKT} holds in a feasible point $\bar{u}\in X$ if and only if, for each $\nu$, the Robinson constraint qualification holds in $\bar{u}^{\nu}$ for the optimization problem
\begin{equation}\label{Eq:GNEP_RCQ}
    \min_{u^{\nu}\in U_{\text{ad}}^{\nu}}\ J_{\nu}(u^{\nu},\bar{u}^{-\nu}) \quad\text{s.t.}\quad
    y(u^{\nu},\bar{u}^{-\nu})\ge \psi.
\end{equation}
This is a standard assumption in the optimal control context which is implied, for instance, by a Slater-type condition \cite{Kanzow2016b}. Assume now that the QVI-tailored RCQ holds, and observe that both $G$ and $D_v G$ (which is constant) are completely continuous by standard results on partial differential equations \cite{Troeltzsch2010}. It then follows from Corollary~\ref{Cor:MoscoRCQ} that the feasible set mapping $\Feas$ is weakly Mosco-continuous in $\bar{u}$.

We now apply the augmented Lagrangian method and choose $\SpaceH:=L^2(\Omega)^N$, $\ConH:=\SpaceH_+$. Keeping in mind the discussion above, we obtain from Theorem~\ref{Thm:Optimality_Cvx} that every weak limit point $\bar{u}$ of the sequence $\{u^k\}$ where RCQ holds is a generalized Nash equilibrium, and that the corresponding subsequence of $\{u^k\}$ actually converges strongly to $\bar{u}$ by Corollary~\ref{Cor:StrMonotone}. The weak-$^*$ convergence of the multiplier sequence follows from Theorem~\ref{Thm:Optimality}.

\begin{figure}\centering
    \subfloat[Solution $\bar{u}^1$]{\includegraphics[scale=0.25]{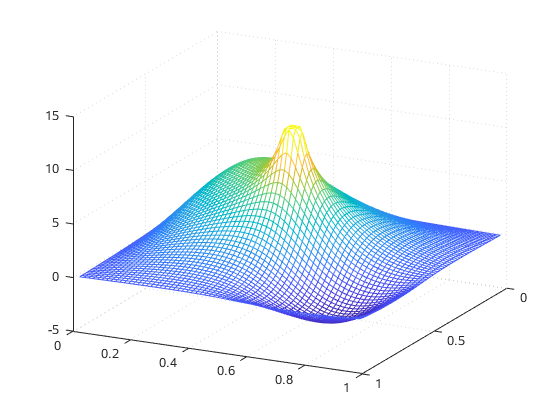}}
    \subfloat[Solution $\bar{u}^2$]{\includegraphics[scale=0.25]{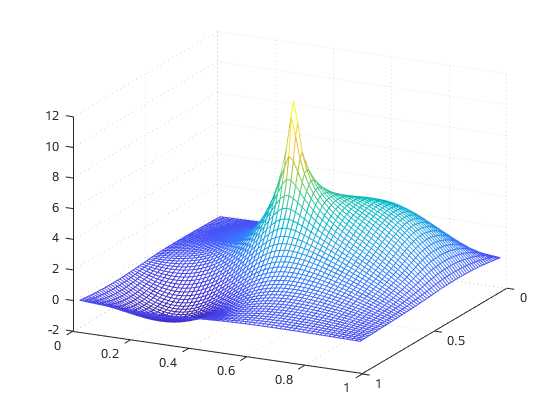}}
    \subfloat[Solution $\bar{u}^3$]{\includegraphics[scale=0.25]{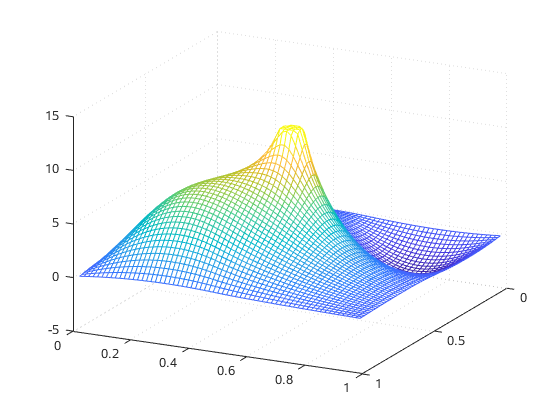}}
    \subfloat[Solution $\bar{u}^4$]{\includegraphics[scale=0.25]{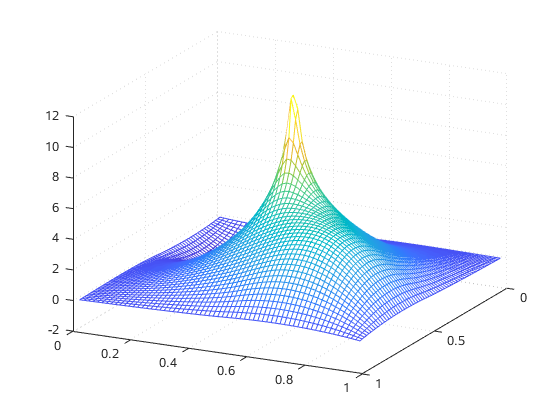}}
    \caption{Numerical results for the optimal control-type GNEP with $n=64$.}
    \label{Fig:OptContGNEP}
\end{figure}

As a numerical application, we consider the first problem presented in \cite{Kanzow2017b}. As before, we discretize the domain $\Omega:=(0,1)^2$ by means of a uniform grid with $n\in\N$ interior points per row or column. The problem in question is a four-player game where $f:= 1$, $U_{\text{ad}}^{\nu}:=\{u^{\nu}\in X: -12\le u^{\nu}\le 12 \}$, and $\alpha:=(2.8859,4.3374,2.5921,3.9481)$. The state constraint $\psi$ is given by
\begin{equation*}
    \psi(x_1,x_2):=\cos \bigl( 5\sqrt{(x_1-0.5)^2+(x_2-0.5)^2} \bigr)+0.1,
\end{equation*}
and the desired states $y_{\textnormal{d}}^{\nu}$ are defined as $y_{\textnormal{d}}^{\nu}:=\xi_{\nu}-\xi_{5-\nu}$, where
\begin{equation*}
    \xi_{\nu}(x_1,x_2):=10^3\max\bigl\{ 0,1-4 \max \{ |x_1-z_{\nu}^1|,|x_2-z_{\nu}^2| \} \bigr\}
\end{equation*}
with $z^1:=(0.25,0.75,0.25,0.75)$ and $z^2:=(0.25,0.25,0.75,0.75)$.

The results of the iteration are displayed in Figure~\ref{Fig:OptContGNEP} and agree with those in \cite{Kanzow2017b}. The corresponding iteration numbers are given as follows.
\begin{center}
\begin{tabular}{r|ccccc}
    $n$ & $16$ & $32$ & $64$ & $128$ & $256$ \\ \hline
	outer it. & 10 & 12 & 12 & 12 & 12 \\
	inner it. & 24 & 29 & 34 & 39 & 43 \\
	$\rho_{\max}$ & $10^8$ & $10^{10}$ & $10^{10}$ & $10^{10}$ & $10^{10}$
\end{tabular}
\end{center}
Once again, we observe that the iteration numbers and final penalty parameters remain nearly constant with increasing $n$.

\section{Final Remarks}\label{Sec:Final}

We have presented a theoretical analysis of quasi-variational inequalities (QVIs) in infinite dimensions and isolated some key properties for their numerical treatment, namely the pseudomonotonicity of the mapping $F$ and the weak Mosco-continuity of the feasible set mapping. Based on these conditions, we have also given an existence result for a rather general type of QVI in Banach spaces. In addition, we have given a formal approach to the first-order optimality (KKT) conditions of QVIs by relating them to those of constrained optimization problems.

Based on the theoretical investigations, we have provided an algorithm of augmented Lagrangian type which possesses powerful global convergence properties. In particular, every weak limit point of the primal sequence is a solution of the QVI under standard assumptions, and the convergence of the multiplier sequence can be established under a suitable constraint qualification (see also Section~\ref{Sec:ApplicSignorini}).

Despite the rather comprehensive character of the paper, there are still multiple aspects which could lead to further developments and research. One important generalization would be to consider not only QVIs of the first kind, but also quasi-variational problems of the second kind or, even more generally, so called quasi-equilibrium problems (see \cite{Blum1994}). In addition, it would be interesting to consider local convergence properties of the augmented Lagrangian method.

\phantomsection
\addcontentsline{toc}{section}{\bibname}
\bibliographystyle{abbrv}
\bibliography{QVI_ALMinf}

\end{document}